\documentclass{article}
\usepackage{amsthm,amsmath,amssymb,amsfonts,graphicx,enumitem}
  \newtheorem{thm}{Theorem}[section]
  \newtheorem{lem}[thm]{Lemma}
  \newtheorem{prop}[thm]{Proposition}
  \newtheorem{cor}[thm]{Corollary}
  \theoremstyle{definition}
  \newtheorem{defn}[thm]{Definition}
  \newtheorem{exm}[thm]{Example}
  \newtheorem{rmk}[thm]{Remark}
  
 \newcommand\ra{\rightarrow}

 \newcommand\mI{\mathcal{I}}

 \newcommand\s{\subseteq}

  \newcommand\lam{\lambda}

\newcommand{\Rad}{\mbox{\rm Rad}}
\newcommand{\Ker}{\mbox{\rm Ker}}
 \numberwithin{equation}{section}
\begin{document}
\title{The Loomis--Sikorski Theorem for $EMV$-algebras }
\author{ Anatolij Dvure\v{c}enskij$^{^{1,2}}$, Omid Zahiri$^{^{3}}$%\footnote{Corresponding author }
\\
{\small\em $^1$Mathematical Institute,  Slovak Academy of Sciences, \v Stef\'anikova 49, SK-814 73 Bratislava, Slovakia} \\
{\small\em $^2$Depart. Algebra  Geom.,  Palack\'{y} Univer. 17. listopadu 12, CZ-771 46 Olomouc, Czech Republic} \\
{\small\em  $^{3}$University of Applied Science and Technology, Tehran, Iran}\\
{\small\tt  dvurecen@mat.savba.sk\quad   zahiri@protonmail.com} }
\date{}
\maketitle



{\small {\it Acknowledgement:} AD is thankful for the support by the grants APVV-16-0073, VEGA No. 2/0069/16 SAV and GA\v{C}R 15-15286S  }

\begin{abstract}
Recently, in \cite{DvZa}, we have introduced $EMV$-algebras which resemble $MV$-algebras but the top element is not guaranteed for them. For $\sigma$-complete $EMV$-algebras, we prove an analogue of the Loomis--Sikorski Theorem showing that every $\sigma$-complete $EMV$-algebra is a $\sigma$-homomorphic image of an $EMV$-tribe of fuzzy sets where all algebraic operations are defined by points. To prove it, some topological properties of the state-morphism space  and the space of maximal ideals are established.
\end{abstract}

\section{Introduction}

Boolean algebras are well-known structures studied over many decades. They describe an algebraic semantics for two-valued logic. In Thirties, there appeared Boolean rings, or equivalently, generalized Boolean
algebras, which have almost Boolean features, but top element is not
assumed. For such structures, Stone, see e.g. \cite[Thm 6.6]{LuZa}, developed a representation of Boolean rings by rings of subsets, and also some logical models with such an incomplete information were established, see \cite{Sto1,Sto2}.

Our approach in \cite{DvZa} was based on analogous ideas: Develop a \L ukasiewicz type algebraic structure with incomplete total information, i.e. find an algebraic semantics very similar to $MV$-algebras with incomplete information, which however in local sense is complete: Conjunctions and disjunctions exist, negation only in local sense, i.e. negation of $a$ in $b$ exists whenever $a\le b$ but total negation of the event $a$ is not assumed.
For such ideas we have introduced in \cite{DvZa} $EMV$-algebras which are locally close to $MV$-algebras, however, the top element is not assumed.

The basic representation theorem says, \cite[Thm 5.21]{DvZa}, that even in such a case, we can find an $MV$-algebra where the original algebra can be embedded as its maximal ideal, i.e. an incomplete information hidden in an $EMV$-algebra is sufficient to find a \L ukasiewicz logical system where top element exists and where all original statements are valid.

Of course, every $MV$-algebra is an $EMV$-algebra ($EMV$-algebras stand for extended $MV$-algebras), and $EMV$-algebras generalize Chang's $MV$-algebras, \cite{Cha}. Nowadays $MV$-algebras have many important applications in different areas of mathematics and logic. Therefore, $MV$-algebras have many different generalizations, like $BL$-algebras, pseudo $MV$-algebras, \cite{georgescu, Dvu2}, $GMV$-algebras in the realm of residuated lattices, \cite{Tsinakis}, etc. In the last period $MV$-algebras are studied also in frames of involutive semirings, see \cite{DiRu}. The presented $EMV$-algebras are another kind of generalizations of $MV$-algebras inspired by Boolean rings.

We note that for $\sigma$-complete $MV$-algebras, a variant of the Loomis--Sikorski Theorem was established in \cite{Mun4,Dvu5,BaWe}. It was shown that, for every $\sigma$-complete $MV$-algebra $M$, there is a tribe of fuzzy sets, which is a $\sigma$-complete $MV$-algebra of $[0,1]$-valued functions with all $MV$-operations defined by points, that can be $\sigma$-homomorphically embedded onto $M$.

The aim of the present paper is to formulated and prove a Loomis--Sikorski type theorem for $\sigma$-complete $EMV$-algebras showing that every $\sigma$-complete $EMV$-algebra is a $\sigma$-homomorphic image of an $EMV$-tribe of fuzzy sets, where all $EMV$-operations are defined by points. To show this, we introduce the hull-kernel topology of the maximal ideals of $EMV$-algebras and the weak topology of state-morphisms which are $EMV$-homomorphisms from the $EMV$-algebra into the $MV$-algebra of the real interval $[0,1]$, or equivalently, a variant of extremal probability measures.

The paper is organized as follows. Section 2 gathers the main notions and results on $EMV$-algebras showing that every $EMV$-algebra without top element can be embed into an $MV$-algebra as its maximal ideal.  Dedekind $\sigma$-complete $EMV$-algebras are studied in Section 3 where also some one-to-one relationships among maximal ideals, maximal filters and state-morphisms are established. In Section 4 we introduce the weak topology of state-morphisms and the hull-kernel topology of maximal ideals.  We show that these spaces are always mutually homeomorphic, locally compact Hausdorff spaces which are compact if and only if the $EMV$-algebra possesses the top element. We prove that if our $EMV$-algebra has no top element, then the state-morphism space of the representing $MV$-algebra is the one-point compactification of the state-morphism space of the original $EMV$-algebra. The Loomis--Sikorski representation theorem will be established in Section 5 together with some topological properties of the state-morphism space and the space of the maximal ideals.

\section{Elements of $EMV$-algebras}%2

An $MV$-{\it algebra} is an algebra $(M;\oplus,^*,0,1)$ (henceforth write simply $M=(M;\oplus,^*,0,1)$) of type $(2,1,0,0)$, where $(M;\oplus,0)$ is a
commutative monoid with the neutral element $0$ and for all $x,y\in M$, we have:
\vspace{1mm}
\begin{enumerate}[nolistsep]
	\item[(i)] $x^{**}=x$;
	\item[(ii)] $x\oplus 1=1$;
	\item[(iii)] $x\oplus (x\oplus y^*)^*=y\oplus (y\oplus x^*)^*$.
\end{enumerate}
\vspace{1mm}

\noindent
In any $MV$-algebra $(M;\oplus,^*,0,1)$, we can define also the following  operations:
\[
x\odot y:=(x^*\oplus y^*)^*,\quad x\ominus y:=(x^*\oplus y)^*.
\]

Then $M$ is a distributive lattice where
$x\vee y=(x\ominus y)\oplus y$ and $x\wedge y=x\odot (x^*\oplus y)$. Note that, for each $x\in M$, $x^*$ is the least element
of the set $\{z\in M\mid x\oplus z=1\}$, i.e.

\begin{equation}\label{eq:MV}
x^*:=\min\{z \in M \mid z\oplus x = 1\}.
\end{equation}
For example, if $(G,u)$ is an Abelian unital $\ell$-group with strong unit $u$, then the interval $[0,u]$ can be converted into an $MV$-algebra as follows: $x\oplus y := (x+y)\wedge u$, $x^*:=u-x$ for all $x,y \in [0,u]$. Then $\Gamma(G,u):=([0,u];\oplus, ^*,0,u)$ is an $MV$-algebra and due to the Mundici result, every $MV$-algebra is isomorphic to some $\Gamma(G,u)$, see \cite{Mun}. For more info about $MV$-algebras, see \cite{CDM}.

An element $a\in M$ is said to be {\it Boolean} or {\it idempotent} if $a\oplus a = a$, or equivalently, $a\vee a^*=1$. The set $B(M)$ of Boolean elements of $M$ forms a Boolean algebra.

Given $a\in B(M)$, we can define a new $MV$-algebra $M_a$ whose universe is the interval $[0,a]$ and the $MV$-operations are inherited from the original one as follows: $M_a=([0,a]; \oplus, ^{*_a},0,a)$, where $x^{*_a}=a\odot x^*$ for each $x \in [0,a]$.
Then we have
$$
x^{*_a}= \min\{z \in [0,a]\colon z\oplus x = a\}, \quad x \in [0,a].
$$
In the paper, we will write also $\lambda_a(x):= x^{*_a}$, $x\in [0,a]$,

Inspired by these properties of $MV$-algebras, in\cite{DvZa}, we have introduced $EMV$-algebras as follows. Let $(M;\oplus,0)$ be a commutative monoid with a neutral element $0$. An element $a \in M$ is said to be an {\it idempotent} if $a\oplus a=a$. We denote by $\mathcal I(M)$ the set of idempotent elements of $M$; clearly $0 \in \mathcal I(M)$, and if $a,b \in I(M)$, then $a\oplus b \in \mathcal I(M)$.

According to \cite{DvZa}, an $EMV$-{\it algebra} is an algebra $(M;\vee,\wedge,\oplus,0)$ of type $(2,2,2,0)$ such that
\begin{enumerate}[nolistsep]
\item[(i)] $(M;\oplus,0)$ is a commutative monoid with a neutral element $0$;
\item[(ii)] $(M;\vee,\wedge,0)$ is a distributive lattice with the bottom element $0$;
\item[(iii)] for each idempotent $a \in \mathcal I(M)$, the algebra $([0,a];\oplus, \lambda_a,0,a)$ is an $MV$-algebra;

\item[(iv)] for each $x \in M$, there is an idempotent $a$ of $M$ such that $x\le a$.
\end{enumerate}
We notify that according to (\ref{eq:MV}), we have for each $a\in \mathcal I(M)$
$$
\lambda_a(x)=\min\{z \in [0,a]\mid z\oplus x = a\}, \quad x\in [0,a].
$$
We note that the existence of a top element in an $EMV$-algebra is not assumed, and if it exists, then $M=(M;\oplus, \lambda_1,0,1)$ is an $MV$-algebra. We note that every $MV$-algebra forms an $EMV$-algebra, every generalized Boolean algebra (or equivalently Boolean ring) is an $EMV$-algebra.

Besides the operation $\oplus$ we can define an operation $\odot$ as follows: Let $x,y \in M$ and let $x,y\le a\in \mathcal I(M)$. Then
$$x\odot y :=\lambda_a(\lambda_a(x)\oplus \lambda_a(y)).
$$
As it was shown in \cite[Lem 5.1]{DvZa}, the operation $\odot$ does not depend on $a\in \mathcal I(M)$. Then we have:
If $x,y \in [0,a]$ for some idempotent $a\in M$, then
\begin{equation}\label{eq:x<y2}
x\odot \lambda_a(y)=x\odot \lambda_a(x\wedge y) \quad \mbox{ and }\quad  x=(x\wedge y) \oplus (x\odot \lambda_a(y)).
\end{equation}

For any integer $n\ge 1$ and any $x$ of an $EMV$-algebra $M$, we can define
$$0.x=0,\quad 1.x = x,\quad (n+1).x= (n.x)\oplus x,$$
and
$$ x^1 =1, \quad x^n=x^{n-1}\odot x, \ n\ge 2,$$
and if $M$ has a top element $1$, we define also $x^0=1$.

We define the classical notions like ideal: An {\it ideal} of an $EMV$-algebra is a non-void subset $I$ of $M$ such that (i) if $x\le y \in I$, then $x\in I$, and (ii) if $x,y \in I$, then $x\oplus y$. An ideal is {\it maximal} if it is a proper ideal of $M$ which is not properly contained in another proper ideal of $M$. Nevertheless $M$ has not necessarily a top element, every $M\ne \{0\}$ has a maximal ideal, see \cite[Thm 5.6]{DvZa}. We denote by $\mathrm{MaxI}(M)$ the set of maximal ideals of $M$. The {\it radical} $\mbox{\rm Rad}(M)$ of $M$, is the intersection of all maximal ideals of $M$, and for it we have
\begin{equation}\label{eq:Rad}
\Rad(M)=\{x\in M\mid x\neq 0,\ \exists\, a\in\mI(M): x\leq a\ \& \ \ (n.x\leq \lam_a(x),\  \forall\, n\in\mathbb{N})\}\cup\{0\}.
\end{equation}

A {\it filter} is a dual notion to ideals, i.e. a non-void subset $F$ of $M$ such that (i) $x\ge y \in F$ implies $x \in F$, and (ii) if $x,y \in F$, then $x\odot y \in F$.

A subset $A\s M$ is called an $EMV$-{\it subalgebra} of $M$ if $A$ is closed
under $\vee$, $\wedge$, $\oplus$ and $0$ and, for each $b\in \mI(M)\cap A$, the set $[0,b]_A:=[0,b]\cap A$ is a subalgebra
of the $MV$-algebra $([0,b];\oplus,\lam_b,0,b)$.  Clearly, the last condition is equivalent to the following condition:
$$\forall\, b\in A\cap \mI(M),\quad \forall\, x\in [0,b]_A,\ \ \min\{z\in [0,b]_A\mid x\oplus z=b\}=\min\{z\in [0,b]\mid x\oplus z=b\},
$$
or equivalently, $x\in [0,b]\cap A$ implies $\lambda_b(x)\in [0,b]\cap A$ whenever $b \in A\cap \mathcal I(M)$.
Let $(M_1;\vee,\wedge,\oplus,0)$ and $(M_2;\vee,\wedge,\oplus,0)$ be $EMV$-algebras. A map $f:M_1\ra M_2$ is called an $EMV$-{\it homomorphism}
if $f$ preserves the operations $\vee$, $\wedge$, $\oplus$ and $0$, and for each $b\in\mI(M_1)$ and for each $x\in [0,b]$, $f(\lam_b(x))= \lam_{f(b)}(f(x))$.

As it was said, it can happen that an $EMV$-algebra $M$ has no top element, however, it can be embedded into an $MV$-algebra $N$ as its maximal ideal as it was proved in \cite[Thm 5.21]{DvZa}:

\begin{thm}\label{th:embed}{\rm [Basic Representation Theorem]}
Every $EMV$-algebra $M$ is either an $MV$-algebra or $M$ can be embedded into an $MV$-algebra $N$ as a maximal ideal of $N$ such that every element $x\in N$ either belongs to the image of the embedding of $M$, or it is a complement of some element $x_0$ belonging to the image of the embedding of $M$, i.e. $x=\lambda_1(x_0)$.
\end{thm}
The $MV$-algebra $N$ from the latter theorem is said to be {\it representing} the $EMV$-algebra $M$. A similar result for generalized Boolean algebras was established in \cite[Thm. 2.2]{CoDa}.

A mapping $s:M \to [0,1]$ is a {\it state-morphism} if $s$ is an $EMV$-homomorphism from $M$ into the $EMV$-algebra of the real interval $[0,1]$ such that there is an element $x\in M$ with $s(x)=1$. We denote by $\mathcal{SM}(M)$ the set of state-morphisms on $M$. In \cite[Thm 4.2]{DvZa} it was shown that if $M\ne \{0\}$, $M$ admits at least one state-morphism. In addition, there is a one-to-one correspondence between state-morphisms and maximal ideals given by a relation: If $s$ is a state-morphism, then $\Ker(s)=\{x \in M \mid s(x)=0\}$ is a maximal ideal of $M$, and conversely, for each maximal ideal $I$ there is a unique state-morphism $s$ on $M$ such that $\Ker(s)=I$.

An $EMV$-algebra $M$ is said to be {\it semisimple} if $\Rad(M)=\{0\}$. Semisimple $EMV$-algebras can be characterized by $EMV$-clans. A system $\mathcal T\subseteq [0,1]^\Omega$ of fuzzy sets of a set $\Omega\ne \emptyset$ is said to be an $EMV$-{\it clan} if
\vspace{1mm}
\begin{enumerate}[nolistsep]
\item[(i)] $0_\Omega \in \mathcal T$ where $0_\Omega(\omega)=0$ for each $\omega \in \Omega$;
\item[(ii)] if $a \in \mathcal T$ is a characteristic function, then (a) $a-f \in \mathcal T$ for each $f\in \mathcal T$ with $f(\omega)\le a(\omega)$ for each $\omega \in \Omega$, (b) if $f,g \in \mathcal T$ with $f(\omega),g(\omega)\le a(\omega)$ for each $\omega \in \Omega$, then $f\oplus g \in \mathcal T$, where $(f\oplus g)(\omega) = \min\{f(\omega)+g(\omega),a(\omega)\}$, $\omega \in \Omega$, and $a$ is a characteristic function from $\mathcal T$;
\item[(iii)] for each $f \in \mathcal T$, there is a characteristic function $a \in \mathcal T$ such that $f(\omega)\le a(\omega)$ for each $\omega \in \Omega$;
\item[(iv)] given $\omega \in \Omega$, there is $f \in \mathcal T$ such that $f(\omega)=1$.
\end{enumerate}

Then $M$ is semisimple iff there is an $EMV$-clan $\mathcal T$ that is isomorphic to $M$, see \cite[Thm 4.11]{DvZa}.

For other unexplained notions and results, please consult with the paper \cite{DvZa}.

\section{Dedekind $\sigma$-complete $EMV$-algebras}%3

In the present section, we study Dedekind $\sigma$-complete $EMV$-algebras and we show a one-to-one correspondence between the set of maximal ideals and the set of maximal filters using the notion of state-morphisms.

We say that an $EMV$-algebra $M$ is {\it Archimedean in the sense of Belluce} if, for each $x,y \in M$ with $n.x \le y$ for all $n\ge 0$, we have $x\odot y =x$. This notion was introduced by \cite{Bel} for $MV$-algebras, see also \cite[p. 395]{DvPu}.

\begin{prop}\label{pr:Bel1}
Let $M$ be an $EMV$-algebra. The following statements are equivalent:
\vspace{1mm}
\begin{itemize}[nolistsep]
\item[{\rm (i)}] $M$ is Archimedean in the sense of Belluce.
\item[{\rm (ii)}] For each $a\in \mathcal I(M)$, the $MV$-algebra $[0,a]$ is Archimedean in the sense of Belluce.
\item[{\rm (iii)}] For each $a \in \mathcal I(M)$, the $MV$-algebra $[0,a]$ is semisimple.
\item[{\rm (iv)}] $M$ is semisimple.
\end{itemize}
\end{prop}

\begin{proof}
(i) $\Rightarrow$ (ii) If $x,y \in [0,a]$, then $x\odot y\in [0,a]$ so that the implication is evident.

(ii) $\Rightarrow$ (i) Let $x,y \in M$ and let $n.x \le y$ for each $n\ge 0$. There is an idempotent $a \in M$ such that $x,y \le a$. Hence $n.x \le y\le a$, so that $x\odot y=x$.

(ii) $\Leftrightarrow$ (iii) It follows from \cite[Thms 31, 33]{Bel}.

(iii) $\Rightarrow$ (iv) We use equation (\ref{eq:Rad}). Assume $x \in \Rad(M)$. By \cite[Thm 5.14]{DvZa}, there is an idempotent $a \in M$ such that $x\le a$ and $n.x \le \lambda_a(x)$. Using Archimedeanicity in the sense of Belluce holding in the $MV$-algebra $[0,a]$, we have $0=x\odot \lambda_a(x)= x$, so that $\Rad(M)=\{0\}$ and $M$ is semisimple.

(iv) $\Rightarrow$ (iii) Let $a$ be an arbitrary idempotent of $M$.  If $I$ is a maximal ideal of $M$, then by \cite[Prop 3.23]{DvZa}, $[0,a]\cap I$ is either $[0,a]$ or a maximal ideal of $[0,a]$. Since $\{0\}=\Rad(M)=\bigcap\{I \mid I\in  \mathrm{MaxI}(M)\}$, we have $\Rad([0,a])\subseteq [0,a]\cap \Rad(M)=\{0\}$ proving $[0,a]$ is a semisimple $MV$-algebra.
\end{proof}

According to the Basic Representation Theorem, Theorem \ref{th:embed}, every $EMV$-algebra $M$ is either an $MV$-algebra or it can be embedded into an $MV$-algebra $N$ as its maximal ideal, so that we can assume that $M$ is an $EMV$-subalgebra of $N$. We define a notion of Dedekind $\sigma$-complete $EMV$-algebras as follows.

We say that an $EMV$-algebra $M$ is {\it Dedekind} $\sigma$-{\it complete} if, for each sequence $\{x_n\}$ of elements of $M$ for which there is an element $x_0\in M$ such that $x_n\le x_0$ for each $n$, $\bigvee_n x_n$ exists in $M$. It is easy to see that $M$ is Dedekind $\sigma$-complete iff $[0,a]$ is a $\sigma$-complete $MV$-algebra for each idempotent $a \in M$.

\begin{lem}\label{le:Bel2}
{\rm (i)} If $x\in M$ is the least upper bound of a sequence $\{x_n\}$ of elements of an $EMV$-algebra $M$, then it is the least upper bound in $N$.

{\rm (ii)} If $\{x_n\}$ has an upper bound $a\in \mathcal I(M)$, then $\bigvee_n x_n$ exists in $M$ if and only if it exists in the $MV$-algebra $[0,a]$. In either case, the suprema coincide.

{\rm (iii)} $M$ is Dedekind $\sigma$-complete if and only if, given a sequence $\{y_n\}$ of elements of $M$, there is $y=\bigwedge_n y_n\in M$.

If $x=\bigvee_n x_n \le a \in \mathcal I(M)$, then
$$\lambda_a(x) =\bigwedge_n \lambda_a(x_n),$$
and if $y = \bigwedge_n y_n$ and $y_n\le a\in \mathcal I(M)$, then
$$\lambda_a(y)=\bigvee_n\lambda_a(y_n).
$$
\end{lem}

\begin{proof}
(i) If $M=N$, the statement is trivial. So let $M$ be a proper $EMV$-algebra, i.e.,  $M\varsubsetneq N$.  Assume that for $y\in N\setminus M$, we have $x_n \le y$ for each $n$. Then $y=y^*_0:=\lambda_1(y_0)$ for some $y_0\in M$, where $1$ is the top element of $N$. We have $x_n \le x\wedge y^*_0 \le x, y^*_0$. Since $M$ is a maximal ideal of $N$, we have $x\wedge y^*_0\in M$ which entails $x\le x\wedge y^*_0\le x$, and finally $x\le y_0^*$ proving $x$ is the least upper bound also in $N$.

(ii) Let $x = \bigvee_n x_n$, and $x \le a\in \mathcal I(M)$. If $y\in [0,a]$ is an upper bound of $\{x_n\}$, then clearly $x\le y$, so that $x$ is also its least upper bound taken in $[0,a]$. Conversely, let $x$ be the least upper bound of $\{x_n\}$ taken in the $MV$-algebra $[0,a]$ and let $y\in M$ be an arbitrary upper bound of $\{x_n\}$. Then $x_n \le y \wedge a\le a$ so that $x\le y\wedge a \le y$.

(iii) Assume $M$ is Dedekind $\sigma$-complete and let $\{y_n\}$ be a sequence of elements of $M$. Since $M$ is a lattice, we can assume $y_{n+1}\le y_n\le y_1$ for each $n\ge 1$. There is an idempotent $a\in M$ such that $y_n\le a$ for each $n\ge 1$. Then $\lambda_a(y_n)\le \lambda_a(y_{n+1})\le a$, so that there is $y_0=\bigvee_n \lambda_a(y_n) \in [0,a]$. We assert $\lambda_a(y_0)=\bigwedge_n y_n$.  Let $y'\le y_n$ for each $n\ge 1$, then $\lambda_a(y_n)\le \lambda_a(y')$ so that $\lambda_a(y')\le y_0$, and $y'=\lambda^2_a(y^*)\le \lambda_a(y_0)$.

Conversely, let every sequence from $M$ have the infimum in $M$. Let $\{x_n\}$ be an arbitrary sequence from $M$ with an upper bound $x_0\in M$; we can assume $x_n \le x_{n+1}$ for each $n\ge 1$. There is an idempotent $a\in M$ such that $x_n\le x_0\le a$. Then $a\ge \lambda_a(x_n)\ge \lambda_a(x_{n+1})\ge \lambda_a(x_0)$, and there is $z_0=\bigwedge_n \lambda_a(x_n)$. As in the previous case, we can show $\lambda_a(z_0)= \bigvee_n x_n$.
\end{proof}

For the next result, we need the following notion. We say that an $EMV$-algebra $M$ satisfies the {\it general comparability property} if it holds for every $MV$-algebra $([0,a]; \oplus,\lambda_a,0,a)$, i.e. if $a \in \mathcal I(M)$ and $x,y \in [0,a]$, there is an idempotent $e$, $e \in [0,a]$ such that $x\wedge e\le y$ and $y\wedge \lambda_a(e) \le x$.

\begin{prop}\label{pr:Bel3}
If an $EMV$-algebra $M$ is Dedekind $\sigma$-complete, then $M$ is a semisimple $EMV$-algebra satisfying the general comparability property, and the set of idempotent elements $\mathcal I(M)$ is a Dedekind $\sigma$-complete subalgebra of $M$.
\end{prop}

\begin{proof}
Let $a\in M$ be an idempotent. Since $M$ is Dedekind $\sigma$-complete, then $[0,a]$ is a $\sigma$-complete $MV$-algebra, and by \cite[Prop 6.6.2]{CDM}, $[0,a]$ is semisimple. Applying Proposition \ref{pr:Bel1}, we conclude that $M$ is semisimple. Using \cite[Thm 9.9]{Goo}, we can conclude that every $MV$-algebra $[0,a]$ satisfies the general comparability property, consequently, so does $M$.

Now let $\{a_n\}$ be a sequence of idempotent elements of $M$ bounded by some element $x$. Clearly, $\{a_n\}$ is bounded by some idempotent $a_0$. Let $a = \bigvee_n a_n$ exists in $M$. For any $n$, let $b_n=a_1\vee\cdots \vee a_n$. Then $a = \bigvee_n b_n$. Using \cite[Prop 1.21]{georgescu}, we have $a\oplus a= a\oplus \Big(\bigvee_n b_n\Big)= \bigvee_n (a\oplus b_n)=\bigvee_n\bigvee_m(b_n \oplus b_m)=\bigvee_n\Big(\bigvee_{m\le n} (b_n\oplus b_m)\vee \bigvee_{m>n} (b_n \oplus b_m)\Big)=\bigvee_n \Big(\bigvee_{m\le n} b_n \vee \bigvee_{m>n}b_m\Big)= \bigvee_n b_n=a$. That is, $a$ is an idempotent of $M$.
\end{proof}

\begin{prop}\label{pr:Bel4}
Let $M$ be an $EMV$-algebra. If $\bigvee_t y_t$ exists in $M$, then for each $x \in M$, $\bigvee_t (x\wedge y_t)$ exists and
\begin{eqnarray*}
x \wedge \bigvee_t y_t&=& \bigvee_t (x \wedge y_t),\\
\Big(\bigvee_t y_t\Big)\odot x &=& \bigvee_t(y_t \odot x),\\
\end{eqnarray*}
where $a \in \mathcal I(M)$ such that $x,\bigvee_t y_t \le a$.
\end{prop}

\begin{proof}
Let $y=\bigvee_t y_t$ exist in $M$. Clearly, $x\wedge y \ge x\wedge y_t$ for each $t$. Now let $z\ge x\wedge y_t$ for each $t$. There is an idempotent $a\in M$ such that $x,y,z \le a$. Then the statement holds in the $MV$-algebra $[0,a]$, see e.g. \cite[Prop 1.18]{georgescu}, so does in $M$.

The second property holds also in the $MV$-algebra $[0,a]$ as it follows from \cite[Prop 1.16]{georgescu}.
\end{proof}

Let $s$ be a state-morphism on $M$. We define two sets
$$
\Ker(s):=\{x \in M\mid s(x)=0\},\quad \Ker_1(s)=\{x\in M\mid s(x)=1\}.
$$
We have the following simple but useful characterization of maximal ideals and maximal filters by state-morphisms.

\begin{lem}\label{le:Bel5}
Let $s$ be a state-morphism on an $EMV$-algebra $M$. Then $\Ker(s)$ is a maximal ideal of $M$ and $\Ker_1(s)$ is a maximal filter of $M$. Conversely, for each maximal ideal $I$ and each maximal filter $F$, there are unique state-morphisms $s$ and $s_1$ on $M$ such that $I=\Ker(s)$ and $F=\Ker_1(s_1)$.
\end{lem}

\begin{proof}
The one-to-one correspondence between $\Ker(s)$ and a maximal ideal $I$ of $M$ was established \cite[Thm 4.2]{DvZa}.

Now we show that $\Ker_1(s)$ is a maximal filter of $M$.  It is clear that $\Ker_1(s)$ is a filter. Let $x \notin \Ker_1(s)$. Then $s(x)<1$ and since $s(x)$ is a real number in the $MV$-algebra of the real interval $[0,1]$, we have that there are an integer $n$ such that $s(x^n)=(s(x))^n=0$ and an idempotent $b \in \mathcal I(M)$ such that $x\le b$ and $s(b)=1$. Then $x^n \oplus \lambda_b(x^n)=b$, so that $\lambda_b(x^n)\in \Ker_1(s)$ which by criterion (ii) of \cite[Prop 5.4]{DvZa} means $\Ker_1(s)$ is a maximal filter.

Now let $F$ be a maximal filter of $M$. Define $I_F=\{\lambda_a(x) \mid x\in F, a \in \mathcal I(M), x\le a\}$. By \cite[Thm 5.6]{DvZa}, $I_F$ is a maximal ideal of $M$ so that, there is a unique state-morphism $s$ such that $\Ker(s)=I_F$. Now let $x \in F$ and let $a$ be an idempotent of $M$ such that $x\le a$ and $s(a)=1$. Then $s(\lambda_a(x))=0$, so that $1=s(a)= s(x\oplus \lambda_a(x))= s(x)$, and $F\subseteq \Ker_1(s)$. The maximality of $F$ and $\Ker_1(s)$ yields $F=\Ker_1(s)$.

If there is another state-morphism $s'$ such that $\Ker_1(s)=F=\Ker_1(s')$, then $\Ker(s)=I_F=\Ker(s')$, which by \cite[Thm 4.3]{DvZa} means $s=s'$.
\end{proof}

\section{Hull-Kernel Topologies and the Weak Topology\\ of State-Morphisms}%4

The present section is devoted to the hull-kernel topology of the set of maximal ideals and the weak topology of the set of state-morphisms. We show that these spaces are homeomorphic, and more information can be derived for $EMV$-algebras with the general comparability property. In addition, using the Basic Representation Theorem, we show that if an $EMV$-algebra $M$ has no top elements, the state-morphism space is only locally compact and not compact, and its one-point compactification is homeomorphic to the
state-morphism space of $N$. The similar property holds for the set of maximal filters of $M$ and $N$, respectively.

We remind that a topological space $\Omega\ne \emptyset$ is \vspace{1mm}

\begin{itemize}[nolistsep]

\item[(i)] {\it regular} if, for each point $\omega\in \Omega$ and any closed subspace $A$ of $\Omega$ not-containing $\omega$, there are two disjoint open sets $U$ and $V$ such that $\omega \in U$ and $A\subseteq V$;

\item[(ii)] {\it completely regular} if, for each non-empty closed set $F$ and each point $a\in \Omega\setminus F$, there is a continuous function $f:\Omega \to [0,1]$ which such that $f(\omega)=1$ for each $\omega \in F$ and $f(a)=0$;

\item[(iii)] {\it totally disconnected} if every two different points are separated by a clopen subset of $\Omega$;

\item[(iv)] {\it locally compact} if every point of $\Omega$ has a compact neighborhood;

\item[(v)] {\it basically disconnected} if the closure of every open $F_\sigma$ subset of $\Omega$ is open.
\end{itemize}
\vspace{1mm}
Of course, (i) implies (ii).
We note that the weak topology of state-morphisms on a $\sigma$-complete $MV$-algebra is basically disconnected, see e.g. \cite[Prop 4.3]{Dvu5}.

On the set $\mathrm{MaxI}(M)$ of maximal ideals of $M$ we introduce the following hull-kernel topology $\mathcal T_M$.

\begin{prop}\label{pr:Bel6}
Let $M$ be an $EMV$-algebra. Given an ideal $I$ of $M$, let
$$
O(I) := \{A \in  \mathrm{MaxI}(M) \mid A \not\supseteq I\},
$$
and let $\mathcal T_M$ be the collection of all subsets of the above form. Then $\mathcal T_M$ defines a topology on $\mathrm{MaxI}(M)$ which is a Hausdorff one.

Given $a \in M$, we set
$$M(a)=\{I \in \mathrm{MaxI}(M) \mid a \notin I\}.
$$
Then $\{M(a)\mid a \in M\}$ is a base for $\mathcal T_M$. In addition,
we have
\vspace{1mm}
\begin{enumerate}[nolistsep]
\item[{\rm (i)}]  $M(0) = \emptyset$;
\item[{\rm (ii)}] $M(a) \subseteq  M(b)$ whenever $a \le b$,
\item[{\rm (iii)}] $M(a \wedge b) = M(a) \cap M(b)$, $M(a \vee b) = M(a)\cup  M(b)$.
\end{enumerate}
Moreover, any closed subset of $\mathcal T_M$ is of the form
$$
C(I) := \{A \in  \mathrm{MaxI}(M) : A \supseteq I\}.
$$
\end{prop}

\begin{proof}
We have (i) $O(\{0\})=\emptyset$, $O(M)=\mathrm{MaxI}(M)$, (ii) if $I\subseteq J$, then $O(I)\subseteq O(J)$, (iii) $\bigcup_\alpha O(I_\alpha)= O(I)$, where $I=\bigvee_\alpha I_\alpha$, and (iv) $\bigcap_{i=1}^n O(I_i)=O(\bigcap_{i=1}^n I_i)$ which implies $\{O(I)\mid I\in \mathrm{Ideal}(M)\}$ defines the topology $\mathcal T_M$ on $\mathrm{MaxI}(M)$

Given $a \in M$, let $I_a$ be the ideal of $M$ generated by $a$. Then $O(I_a)=M(a)$. Since $O(I)=\bigcup\{M(a)\mid a \in I\}$, we see that $\{M(a) \mid a \in M\}$ is a base for $\mathcal T_M$.

To see that $M(a)\cap M(b)=M(a\wedge b)$, we have trivially $M(a)\cap M(b)\supseteq M(a\wedge b)$. Let $A \in M(a)\cap M(b)$ and let $A \notin M(a\wedge b)$. Then $a\wedge b \in A$ and since $A$ is prime, either $a \in A$ or $b\in A$ which is impossible. Then $A \in M(y)$ and $B \in M(x)$.

{\it Hausdorffness.}
Let $A$ and $B$ be two maximal ideals of $M$, $A\ne B$. There are $x \in A\setminus B$ and $y \in B\setminus A$. Then $x\wedge y \in A\cap B$. Let $a$ be an idempotent of $M$ such that $x,y \le a$. Then $x\odot \lambda_a(y) \in [0,a]$. Since $x= (x\odot \lambda_a(y)) \oplus (x\wedge y)$, we see that $x\odot \lambda_a(y)\in A \setminus B$. In a similar way, we have $y\odot \lambda_a(x)\in B \setminus A$. Due to $(x\odot \lambda_a(y))\wedge (y\odot \lambda_a(x))=0$, we have also $A \in M(y\odot \lambda_a(x))$ and $B \in M(x\odot \lambda_a(y))$ and $M(y\odot \lambda_a(x))\cap M(x\odot \lambda_a(y))= M((x\odot \lambda_a(y))\wedge (y\odot \lambda_a(x)))=M(0)=\emptyset$.
\end{proof}

\begin{lem}\label{le:Bel7}
Let $M$ be an $EMV$-algebra. Then
\vspace{1mm}
\begin{itemize}[nolistsep]
\item[{\rm(i)}] If $O(I)= O(M)$, then $I=M$.
\item[{\rm(ii)}] $M(a)=M(0)$ if and only if $a\in \Rad(M)$.
\item[{\rm(iii)}]  If, for some $a\in \mathcal I(M)$, we have $M(a)=O(M)$, then $a$ is the top element of $M$ and $M$ is an $MV$-algebra.
\item[{\rm(iv)}] If, for some $x \in M$, we have $M(x)=O(M)$, then $M$ has a top element.
\item[{\rm(v)}] The space $\mathrm{MaxI}(M)$ is compact if and only if $M$ has the top element.
\end{itemize}
\end{lem}

\begin{proof}
(i) Assume $I$ is a proper ideal of $M$. There is a maximal ideal $A$ of $M$ containing $I$, then $A\notin O(I)=O(M)$ which yields a contradiction with $A \in O(M)$.

(ii) It follows from definition of $\Rad(M)$.

(iii) Let $a$ be an idempotent and let $I_a$ be the ideal of $M$ generated by $a$. From (i), we conclude $I_a=M$. Hence, if $x \in M$, then $x\in I_a$ and henceforth, there is an integer $n$ such that $x\le n.a=a$, i.e., $a$ is the top element of $M$.

(iv) Let $I_x$ be the ideal of $M$ generated by $x$. There is an idempotent $a$ of $M$ such that $x\le a$. We assert $a$ is the top element of $M$. Indeed, from (i), we have $I_x=M$, i.e. for any $z\in M$, there is an integer $n$ such that $z\le n.x$. But then $z\le n.a=a$.

(v)  Let $\mathrm{MaxI}(M)$ be a compact space. Since $\{M(x)\mid x \in M\}$ is an open covering of $\mathrm{MaxI}(M)$, there are finitely many elements $x_1,\ldots,x_n \in M$ such that $\bigcup_{i=1}^n M(x_i)=O(M)$, so that if $x_0 = x_1\vee \cdots \vee x_n$, then $M(x_0)=O(M)$ which by (iv) means that $x_0$ is the top element of $M$.

Conversely, if $M$ has the top element, then $M$ is in fact an $MV$-algebra, and the compactness of $\mathrm{MaxI}(M)$ is well known, see e.g. \cite[Prop 7.1.3]{DvPu}, \cite[Cor 12.19]{Goo}.
\end{proof}

We say that a net $\{s_\alpha\}_\alpha$ of state-morphisms on $M$ {\it converges weakly } to a state-morphism $s$ on $M$,
%and we write $\{s_\alpha\}_\alpha\to s$,
if $\lim_\alpha s_\alpha(a)=s(a)$. Hence, $\mathcal{SM}(M)$ is a subset of $[0,1]^M$ and if we endow $[0,1]^M$ with the product topology which is a compact Hausdorff space, we see that the weak topology, which is in fact a relative topology (or a subspace topology) of the product topology of $[0,1]^M$, yields a non-empty Hausdorff topological space whenever $M\ne \{0\}$; if $M=\{0\}$, the set $\mathcal{SM}(M)$ is empty. In addition, the system of subsets of $\mathcal{SM}(M)$ of the form $S(x)_{\alpha,\beta}=\{s \in \mathcal{SM}(M) \mid \alpha<s(x)<\beta\}$, where $x\in M$ and $\alpha < \beta$ are real numbers, forms a subbase of the weak topology of state-morphisms.

We note that $\mathcal{SM}(M)$ is closed in the product topology
whenever $M$ has a top element. In general, it is not closed because if, for a net $\{s_\alpha\}_\alpha$ of state-morphisms, there exists $s(a)=\lim_\alpha s_\alpha(a)$ for each $a\in M$, then $s$ preserves $\oplus,\vee,\wedge$, but it can happen that there is no guarantee that there is $x\in M$ such that $s(x)=1$ as  the following examples shows.

\begin{exm}\label{ex:contra}
Let $\mathcal T$ be the set of all finite subsets of the set $\mathbb N$ of natural numbers. Then $\mathcal T$ is a generalized Boolean algebra having no top element, and $\mathcal{SM}(\mathcal T)=\{s_n \mid n \in \mathbb N\}$, where $s_n(A)=\chi_A(n)$, $A \in \mathcal T$. However,  $s(A)=\lim_ns_n(A)=0$ for each $A \in \mathcal T$, so that $s$ is no state-morphism.
\end{exm}

Therefore, a non-empty set $X$ of state-morphisms is closed iff, for each net of states $\{s_\alpha\}_\alpha$ of state-morphisms from $X$, such that there exists $s(x)=\lim_\alpha s_\alpha(x)$ for each $x \in M$, then $s$ is a state-morphism on $M$ and $s$ belongs to $X$.

We note that if $x\in M$, then the function $\hat x:\mathcal{SM}(M)\to [0,1]$ defined by
$$\hat x(s):=s(x), \quad s \in \mathcal{SM}(M),
$$
is a continuous function on $\mathcal{SM}(M)$. We denote by $\widehat M=\{\hat x\mid x \in M\}$.

According to Basic Representation Theorem \ref{th:embed}, every $EMV$-algebra $M$ is either an $MV$-algebra or it can be embedded into an $MV$-algebra $N$ as its maximal ideal, so that we can assume that $M$ is an $EMV$-subalgebra of $N$ and $N=\{x\in N \mid \text{ either } x \in M \text{ or } \lambda_1(x) \in M\}$. If $M$ is a proper $EMV$-algebra, i.e. it does not contain any top element, the state-morphism space $\mathcal{SM}(N)$ can be characterized as follows.

\begin{prop}\label{pr:corr1}
Let $M$ be a proper $EMV$-algebra and, for each $x \in M$, we put $x^*=\lambda_1(x)$.
Given a state-morphism $s$ on $M$, the mapping $\tilde s: N \to [0,1]$, defined by
\begin{equation}\label{eq:state}
\tilde s(x)=\begin{cases}
s(x) & \text{ if } x\in M,\\
1-s(x_0) & \text{ if }  x=x^*_0,\ x_0 \in M,
\end{cases} \quad x \in N,
\end{equation}
is a state-morphism on $N$, and the mapping $s_\infty: N \to [0,1]$ defined by $s_\infty(x)=0$ if $x\in M$ and $s_\infty(x)=1$ if $x\notin M$, is a state-morphism on $N$.
Moreover, $\mathcal{SM}(N)=\{\tilde s \mid s \in \mathcal{SM}(M)\}\cup\{s_\infty\}$ and $\Ker(\tilde s) = \Ker(s) \cup \Ker_1^*(s)$, $s \in \mathcal{SM}(M)$, where $\Ker^*_1(s)=\{\lambda_1(x)\mid x \in \Ker_1(s)\}$.

A net $\{s_\alpha\}_\alpha$ of state-morphisms on $M$ converges weakly to a state-morphism $s$ on $M$ if and only if $\{\tilde s_\alpha\}_\alpha$ converges weakly on $N$ to $\tilde s$.
\end{prop}

\begin{proof}
Assume that $N=\Gamma(G,u)$ for some unital Abelian $\ell$-group $(G,u)$. Then $1=u$ and $x^*=\lambda_1(x)=u-x$, where $-$ is the subtraction taken from the $\ell$-group $G$.

Take $s \in \mathcal{SM}(M)$. We have $\tilde s(1)=1$. If $x,y \in M$, then $\tilde s(x\oplus y)= \tilde s(x)\oplus \tilde s(y)$. If $x=x^*_0$, $y=y^*_0$ for $x_0,y_0\in M$, then $x\oplus y = (x_0\odot y_0)^*$, so that $\tilde s(x\oplus y)=1-\tilde s(x_0\odot y_0)=(1-s(x_0))\oplus (1-s(y_0))=\tilde s(x)\oplus \tilde s(y)$. Finally, if $x=x_0$, $y=y^*_0$ for $x_0,y_0\in M$, there exists an idempotent $b\in \mathcal I(M)$ such that $x_0,y_0 \le b$ and $s(b)=1$. Since $x\oplus y =x_0\oplus y_0^*= (y_0\odot x^*_0)^*= (y_0\odot \lambda_b(x_0))^*$ which yields $\tilde s(x\oplus y)=1-s(y_0\odot \lambda_b(x_0))= 1-(s(y_0)\odot (s(b)-s(x_0)) =(1-s(y_0))\oplus s(x_0) = \tilde s(x)\oplus s(y)$. Whence, $\tilde s$ is a state-morphism on $N$.

It is easy to verify that $s_\infty$ is a state-morphism on $N$. We note that the restriction of $s_\infty$ onto $M$ is not a state-morphism on $M$ because it is the zero function on $M$.

We note that $$\mathcal I(N)=\{x \in N \mid \text{ either } x\in \mathcal I(M) \text{ or } x^* \in \mathcal I(M)\}.$$

Let $s$ be a state-morphism on $N$. We have two cases: (i) There is an idempotent $a \in M$ such that $s(a)=1$, then the restriction $s_0$ of $s$ onto $M$ is a state-morphism on $M$, so that $s=\widetilde{s_0} \in \mathcal {SM}(N)$. (ii) For each idempotent $a\in M$, $s(a)=0$. Since given $x\in M$, there is an idempotent $a\in \mathcal I(M)$ with $x\le a$, we have $s(x)=0$ for each $x \in M$ which says $s=s_\infty$.

The last assertions are evident.
\end{proof}

The latter proposition can be illustrated by the following example:

\begin{exm}\label{ex:LS10}
Let $\mathcal T$ be the system of all finite subsets of the set $\mathbb N$ of integers. Then $\mathcal T$ is an $EMV$-algebra that is a generalized Boolean algebra of subsets, $\mathcal T$ has no top element,  $\mathcal{SM}(\mathcal T)=\{s_n\mid n \in \mathbb N\}$ where $s_n=\chi_A(n)$, $A \in \mathcal T$.
If we define $\mathcal N$ as the set of all finite or co-finite subsets of $\mathbb N$, $\mathcal N$ is an $MV$-algebra such that $\mathcal N=\{A \subseteq \mathbb N \mid \text{ either } A \in \mathcal T \text{ or } A^c \in \mathcal T\}$, and $\mathcal N$ is representing $\mathcal T$. Then $\mathcal{SM}(\mathcal N)=\{\tilde s_n\mid n \in \mathbb N\} \cup \{s_\infty\}$, where $\tilde s_n = \chi_A(n)$, $A\in \mathcal N$, and $s_\infty(A) =0$ if $A$ is finite and $s_\infty(A)=1$ if $A$ is co-finite.
In addition, $\lim_n s_n(A)=0$ for each $A \in \mathcal T$ and $\lim_n \tilde s_n(A)=s_\infty(A)$, $A \in \mathcal N$.
\end{exm}

\begin{rmk}\label{rm:Bel}
Since a net $\{s_\alpha\}_\alpha$ of state-morphisms of $M$ converges weakly to a state-morphism $s\in \mathcal{SM}(M)$ iff $\{\tilde s_\alpha\}_\alpha$ converges weakly on $N$ to $\tilde s$, the mapping $\phi:\mathcal{SM}(M)\to \mathcal{SM}(N)$, defined by $\phi(s)=\tilde s$, $s \in \mathcal{SM}(M)$, is injective and continuous, $\phi(\mathcal{SM}(M))$ is open,  but $\phi$ is not necessarily closed, see Example \ref{ex:LS10}. We have $\phi$ is closed iff $M$ possesses a top element.
\end{rmk}

\begin{proof}
If $x\in M$, then $S_N(x)=\{s\in \mathcal{SM}(N)\mid s(x)>0\}= \widetilde{S(x)}:=\{\tilde s \mid s\in S(x)\}$, where $S(x)=\{s\in \mathcal{SM}(M)\mid s(x)>0\}$. Clearly $s_\infty \notin  S_N(x)$ and $S_N(x)$ is an open set of $\mathcal{SM}(N)$. Therefore, for each $\tilde s$, there is an open set of $\mathcal{SM}(N)$, namely $S_N(x)$, which contains $\tilde s$ and $\tilde s\in S_N(x)\subseteq \phi(X)$.  Whence $\phi(X)$ is open in $\mathcal{SM}(N)$.

If $M$ has a top element, then $N=M$ and $\phi$ is the identity, so it is closed and open as well. Conversely, let $\phi$ be closed, then $\phi(X)$ is closed and compact, where $X=\mathcal{SM}(M)$.

Hence, for each open subset $O$ of $\mathcal{SM}(M)$, we have $\phi(O)=\phi(X\setminus C)= \phi(X)\setminus \phi(C)$, where $C$ is a closed subset of $\mathcal{SM}(M)$, so that $\phi$ is an open mapping. Now let $\{O_\alpha\mid \alpha\in A\}$ be an open covering of $X$, then $\phi(X)=\phi(\bigcup_\alpha O_\alpha)=\bigcup_\alpha \phi(O_\alpha)$, and the compactness of $\phi(X)$ yields $\phi(X)=\bigcup_{i=1}^n \phi(O_{\alpha_i})$, so that $X=\bigcup_{i=1}^n O_{\alpha_i}$ which says
$\mathcal{SM}(M)$ is compact. Since $X = \bigcup \{S(x)\mid x \in M\}$, there are finitely many elements $x_1,\ldots,x_k\in M$ such that $X=\bigcup_{i=1}^k S(x_i)=S(x_0)$, where $x_0=x_1\vee\cdots\vee x_k$. If $I_{x_0}$ is the ideal of $M$ generated by $x_0$, then $S(x_0)=\{s\in \mathcal{SM}(M)\mid \Ker(s)\varsupsetneq I_{x_0}\}$, so that $O(I_{x_0})=O(M)= M(x_0)$ which, by Lemma \ref{le:Bel7}(iv), gives $M$ has a top element.
\end{proof}

\begin{prop}\label{pr:Bel8}
Let $M$ be an $EMV$-algebra and $X$ be a non-empty subspace of state-morphisms on $M$ that is closed in the weak topology of state-morphisms. Let $t$ be a state-morphism such that $t \notin X$.
There exists an $a \in M$ such that $t(a) > 1/2$ while $s(a) < 1/2$ for all $s \in X$. Moreover, the element $a\in M$ can be chosen such that $t(a)=1$ and $s(a)=0$ for each $s \in X$.

In particular, the space $\mathcal{SM}(M)$ is completely regular.
\end{prop}

\begin{proof}
(1) Let $t$ be a state-morphism such that $t \notin X$.
We assert that there exists an $a \in M$ such
that $t(a) > 1/2$ while $s(a) < 1/2$ for all $s \in X.$

Indeed, set $A= \{a\in M:\ t(a) > 1/2\}$, and for all $a \in A$,
let
$$W(a) :=\{s\in \mathcal{SM}(M))\mid\ s(a) < 1/2\},
$$
which is an open subset of $\mathcal{SM}(M)$. We note that $A\ne \emptyset$ and $A$ is downward directed and closed under $\oplus$.

We assert that these open subsets cover $X$. Consider any $s \in
X$. Since $\Ker(s)$ and $\Ker(t)$ are non-comparable subsets of
$M$, there exists $x \in \Ker(t)\setminus \Ker(s).$ Hence $t(x) =
0$ and $s(x) > 0$. Choose an idempotent $b \in M$ such that $x\le b$ and $t(b)=1$. There exists an integer $n \ge 1$ such that
$s(n.x) > 1/2$. Since there is also an integer $k$ such that $s(k.x)= k.s(x)=1$ and $k.x\le b$, we conclude $s(b)=1$.
Due to $t$ is a state-morphism, we have
$t(n.x) = 0.$ Putting $a = \lambda_b(n.x)$, we have $t(a) = 1
> 1/2$ and $s(a) < 1/2.$ Therefore $\{W(a)\mid a \in A\}$ is an
open covering of $X$.

(i) If $M$ has a top element, the state-morphism space $\mathcal{SM}(M)$ is compact and Hausdorff, so that $X$ is compact, and $X \subseteq W(a_1)\cup\cdots\cup W(a_n)$ for some $a_1,\ldots, a_n \in A$.

(ii) If $M$ has no top element,
embed $M$ into the $MV$-algebra $N$ as its maximal ideal. Since $s(1)=1$ for each state-morphism $s$ on $N$, we see that $\mathcal{SM}(N)$ is a compact set in the product topology, consequently, it is compact in the weak topology of state-morphisms on $N$. The  mapping $\phi: \mathcal{SM}(M)\to \mathcal{SM}(N)$ defined by $\phi(s)=\tilde s$, where $\tilde s$ is defined through (\ref{eq:state}), is by Proposition \ref{pr:corr1} injective and continuous.

We assert the set $\phi(X)\cup\{s_\infty\}$ is a compact subset of $\mathcal{SM}(N)$. Indeed, let $\{s_\alpha\}_\alpha$ be a net of state-morphisms from $\phi(X)\cup\{s_\infty\}$. Since $\mathcal{SM}(N)$ is compact, there is a subnet $\{s_{\alpha_\beta}\}_\beta$ of the net $\{s_\alpha\}_\alpha$ converging weakly to a state-morphism $s$ on $N$. If $s=s_\infty$, $s\in \phi(X)\cup \{s_\infty\}$. If $s\ne s_\infty$, there is a state-morphism $s_0\in \mathcal{SM}(M)$ such $s=\tilde s_0$. Then there is $\beta_0$ such that for each $\beta>\beta_0$, $s_{\alpha_\beta}\in X$.  Therefore, $s_0\in X$ and $s=\phi(s_0)\in \phi(X)\cup\{s_\infty\}$. We note that $\tilde t \notin \phi(X)\cup\{s_\infty\}$.

For each $a\in A$, let $\widetilde W(a):=\{s\in \mathcal{SM}(N) \mid s(a)< 1/2\}$. Then $\tilde t(a)=t(a)>1/2$ and $0=s_\infty(a)<1/2$, so that $s_\infty \in \widetilde W(a)$ for each $a \in A$. Then $\{\widetilde W(a)\mid a \in A\}$ is an open covering of the compact set $\phi(X)\cup \{s_\infty\}$. There are $a_1,\ldots,a_n \in A$ such that $\phi(X)\cup\{s_\infty\} \subseteq \widetilde W(a_1)\cup \cdots \cup \widetilde W(a_n)$, consequently, $X \subseteq W(a_1)\cup \cdots \cup W(a_n)$. Put $a = a_1 \wedge \cdots \wedge a_n$. Then $a \in A$ and for each $s \in X$, we have $s(a) \le s(a_i) < 1/2$ for $i =1,\ldots, n$, which proves $X \subseteq W(a)$, i.e., $s(a) <  1/2$ for all $s \in X$.

(2) By the first part of the present proof, there exists an $a \in
M$ such that $t(a) > 1/2$ while $s(a) < 1/2$ for all $s \in X$. In addition, there is an idempotent $b$ of $M$ with $a\le b$ and $t(b)=1$.
Then $t(a\wedge \lambda_b(a)) = t(\lambda_b(a))$ and $t(a \odot\lambda_b(a \wedge \lambda_b(a))) = t(a) - t(a\wedge \lambda_b(a))= t(a) - t(\lambda_b(a)) = 2t(a) - 1 > 0$.

Now let $s$ be an arbitrary element of $X$. If $s(a)=0$,  then $s(a\odot\lambda_b(a\wedge \lambda_b(a))) = 0$. If $s(a)>0$, there is an integer $m_s$ such that $s(m_s.a)=m_s.s(a)=1$ and since $m_s.a\le m_s.b=b$, we have $s(b)=1$. Hence, $s(a\wedge \lambda_b(a))=s(a)$, so that $s(a\odot\lambda_b(a\wedge \lambda_b(a))) = s(a)
- s(a\wedge \lambda_b(a)) = 0$. In any case, the element $a \odot \lambda_b(a \wedge \lambda_b(a))$ is an element of $\bigcap\{\Ker(s) \mid s \in X\}$ for which $t(a\odot \lambda_b(a \wedge \lambda_b(a))) > 0$.

(3) From (1) and (2), we have concluded that if we use (\ref{eq:x<y2}), then $a \odot \lambda_b(a \wedge \lambda_b(a))=a\odot a$ and $s(a\odot a)=0$ for each $s \in X$. In addition, $t(a\odot a) >0$. There is an integer $r$ such that $t(r.(a\odot a))=r.t(a\odot a) =1$ and $s(r.(a\odot a))=0$ for each $s\in X$. Hence, for $x=r.(a\odot a)$, we have $\hat x(X)=0$ and $\hat x(t)=1$. Consequently, for the continuous function $f$ on $\mathcal{SM}(M)$ defined by $f(s)=1-\hat x(s)$, we have $f(X)=1$ and $f(t)=0$, so that $\mathcal{SM}(M)$ is completely regular.
\end{proof}

\begin{thm}\label{th:Bel9}
Let $M$ be an $EMV$-algebra. The mapping $\theta: \mathcal{SM}(M) \to \mathrm{MaxI}(M)$, defined by $s\mapsto \Ker(s)$, is a homeomorphism. In addition, the following statements are equivalent: \vspace{1mm}

\begin{itemize}[nolistsep]
\item[{\rm (i)}] $M$ has a top element.
\item[{\rm (ii)}] $\mathcal{SM}(M)$ is compact in the weak topology of state-morphisms
\item[{\rm (iii)}] $\mathrm{MaxI}(M)$ is compact in the hull-kernel topology.
\end{itemize}
\end{thm}

\begin{proof}
Define a mapping $\theta$ on the set of state-morphisms  $\mathcal{SM}(M)$ with values in $\mathrm{MaxI}(M)$ as follows $\theta(s)= \Ker(s)$, $s \in \mathcal{SM}(M)$. By \cite[Thm 4.2]{DvZa}, $\theta$ is a bijection.
Let $C(I)$ be any closed subspace of $\mathrm{MaxI}(M)$. Then
$$
\theta^{-1}(C(I))=\{s \in \mathcal{SM}(M) \mid s(x)=0 \text{ for all } x \in I\}
$$
which is a closed subset of $\mathcal{SM}(M)$. Therefore, $\theta$ is continuous.

Given a non-empty subset $X$ of $\mathcal{SM}(M)$, we set
$$\Ker(X):=\{x \in M \mid s(x)=0 \text{ for all } s \in X\}.
$$
Then $\Ker(X)$ is an ideal of $M$. If, in addition, $X$ is a closed subset of $\mathcal{SM}(M)$, we assert
\begin{equation}\label{eq:Ker}
\theta(X)=C(\Ker(X)).
\end{equation}

The inclusion $\theta(X)\subseteq C(\Ker(X))$ is evident. By Proposition \ref{pr:Bel8}, if $t\notin X$, there is an element $a\in M$ such that $s(a)=0$ for each $s\in X$ and $t(a)=1$. Consequently, $t \notin X$ implies $\theta(t) \notin C(\Ker(X))$, and $C(\Ker(X)) \subseteq \theta(X)$.
As a result, we conclude $\theta$ is a homeomorphism.

(i) $\Rightarrow$ (ii) If $1$ is a top element of $M$, then $s(1)=1$ for each state-morphism $s$, therefore, $\mathcal{SM}(M)$ is a closed subspace of $[0,1]^M$, consequently, it is compact.

(ii) $\Rightarrow$ (iii) Let $\{O_\alpha\}$ be an open cover of $\mathrm{MaxI}(M)$. It is enough to take a cover of the form $\{O(x_\alpha)\}$. Then $\mathcal{SM}(M)=\theta^{-1}(\mathrm{MaxI}(M))=\bigcup_\alpha \theta^{-1}(O(x_\alpha))$. Hence, there are finitely many indices $\alpha_1,\ldots, \alpha_n$ such that $\mathcal{SM}(M)=\bigcup_{i=1}^n \theta^{-1}(O(x_{\alpha_i}))$ and consequently, $\bigcup_{i=1}^n O(x_{\alpha_i})$, which entails $\mathrm{MaxI}(M)$ is compact.

(iii) $\Leftrightarrow$ (i) It was proved in Lemma \ref{le:Bel7}(v).
\end{proof}

\begin{thm}\label{th:Bel10}
Let $M$ be an $EMV$-algebra with the general comparability property. Then the mapping $\xi: \mathrm{MaxI}(M) \to \mathrm{MaxI}(\mathcal I(M))$ defined by $\xi(A)=A\cap \mathcal I(M)$, $A \in \mathrm{MaxI}(M)$, is a homeomorphism.

In addition, the spaces $\mathcal{SM}(\mathcal I(M))$, $\mathcal{SM}(M)$, $\mathrm{MaxI}(\mathcal I(M))$, and $\mathrm{MaxI}(\mathcal I(M))$ are mutually homeomorphic topological spaces.

Any of the topological spaces is compact if and only if $M$ has a top element.
\end{thm}

\begin{proof}
Let $I$ be any ideal of $\mathcal I(M)$, and let $\hat I$ be the ideal of $M$ generated by $I$. Then (i) $I= \hat I\cap \mathcal I(M)$, (ii) $I\subseteq J$ iff $\hat I\subseteq \hat J$, (iii) if $\hat I$ is a maximal ideal $M$, then so is $I$ in $\mathcal I(M)$ (if $I$ is maximal, then $\hat I$ is not necessarily maximal in $M$), and (iv) if $A$ is a maximal ideal of $M$ such that $A \supseteq \hat I$, then $A\cap \mathcal I(M)=I$ (see \cite[Thm 3.24]{DvZa}).

The mapping $\xi: A \mapsto A \cap \mathcal I(M)$, $A \in \mathrm{MaxI}(M)$, gives an ideal of $\mathcal I(M)$ which is prime because $A$ is prime. Then $\xi(A)$ has to be a maximal ideal of $\mathrm{MaxI}(\mathcal I(M))$. In fact, if $a,b \notin \xi(A)$, $a\le b$, then $b=a\vee \lambda_b(a)$, so that $a\wedge \lambda_b(a)=0$ and $\lambda_b(a)\in A\cap \mathcal I(M)$.
Due to \cite[Thm 4.4]{DvZa}, the mapping $\xi$ is injective, and in view of \cite[thm 4.3]{DvZa}, $\xi$ is invertible, i.e. given maximal ideal $I$ of $\mathcal I(M)$, there is a unique extension of $I$ onto a maximal ideal $A$ of $M$ such that $\xi(A)=I$.

Now let $I$ be an ideal of $\mathcal I(M)$. We assert
$$
\xi^{-1}(C(I))=C(\hat I).
$$
Indeed, if $A$ is a maximal ideal of $\mathcal I(M)$ such that $A\supseteq I$, then $\xi^{-1}(A)\supseteq \hat I$. Conversely, if $A$ is a maximal ideal of $M$ such that $A\supseteq \hat I$, then $\xi(A)\supseteq \hat I\cap \mathcal I(M)=I$. As a result, we have $\xi$ is continuous.

According to Theorem \ref{th:Bel9}, the spaces $\mathcal{SM}(M)$ and $\mathrm{MaxI}(M)$ are homeomorphic; the mapping $\theta: s\mapsto \Ker(s)$, $s \in \mathcal{SM}(M)$, is a homeomorphism. Similarly, $\mathcal{SM}(\mathcal I(M))$ and $\mathrm{MaxI}(\mathcal I(M))$ are homeomorphic under the homeomorphism $\theta_0(s)=\Ker(s)$, $s \in \mathcal{SM}(\mathcal I(M))$. If we define $\eta= \theta_0^{-1}\circ \xi\circ \theta$, then $\eta$ is a bijective  mapping from $\mathcal{SM}(M)$ onto $\mathcal{SM}(\mathcal I(M))$ such that if $s$ is a state-morphism of $M$, then $\eta(s)=s_0:=s_{|\mathcal I(M)}$, the restriction of $s$ onto $\mathcal I(M)$. Conversely, if $s$ is a state-morphism on $\mathcal I(M)$, then $\eta^{-1}(s)=\bar{s}$, the unique extension of $s$ onto $M$. We see that $\eta$ is a continuous mapping.

Now take an $MV$-algebra $N$ such that $M$ can be embedded into $N$ as its maximal ideal, and every element $x$ of $N$ either belongs to $M$ or $\lambda_1(x) \in M$. Given a state-morphism $s$ on $M$, let $\tilde s$ be its extension to $N$ defined by (\ref{eq:state}). According to the proof of Proposition \ref{pr:Bel8}, the mapping $\phi: \mathcal{SM}(M)\to \mathcal{SM}(N)$ given by $\phi(s)=\tilde s$ is injective and continuous, and a net $\{s_\alpha\}_\alpha$ of states of $\mathcal{SM}(M)$ converges weakly to a state-morphism $s\in \mathcal{SM}(M)$ iff $\{\phi(s_\alpha)\}_\alpha$ converges weakly to the state-morphism $\phi(s)$ on $N$.

Take a closed non-void subset $X$ of state-morphisms on $M$, then $\phi(X)$ is a closed subset of $\mathcal{SM}(N)$, consequently, $\phi(X)$ is compact. Let $\{s_\alpha\}_\alpha$ be a net of state-morphisms from $X$ and let its restriction $\{\bar s_\alpha\}_\alpha$ to $\mathcal I(M)$ converges
weakly to a state-morphism $s_0$ on $\mathcal I(M)$. Since the net $\{\tilde s_\alpha\}_\alpha$ is from the compact $\phi(X)$, there is a subnet $\{\bar s_{\alpha_\beta}\}_\beta$ of the net $\{\tilde s_\alpha\}_\alpha$ which converges weakly to a state-morphism $t\in \phi(X)$ on $N$, i.e. $\lim_\beta \tilde s_{\alpha_\beta}(x) = t(x)$ for each $x\in N$. Since $s_\infty \notin \phi(X)$, there is a state-morphism $s\in X$ with $\tilde s=t$. Then $\lim_\beta s_{\alpha_\beta}(x) = s(x)$ for each $x \in M$. In particular, this is true for each $x \in \mathcal I(M)$, so that $\eta(s) = s_0$. In other words, we have proved that $\eta$ is a closed mapping, and whence, $\eta$ is a homeomorphism.

Since $\xi= \theta_0 \circ \eta \circ\theta^{-1}$, we see that $\xi$ is a homeomorphism, and in view of Theorem \ref{th:Bel9}, the spaces $\mathcal{SM}(\mathcal I(M))$, $\mathcal{SM}(M)$, $\mathrm{MaxI}(\mathcal I(M))$, and $\mathrm{MaxI}(\mathcal I(M))$ are mutually homeomorphic topological spaces.

Consequently, according to Theorem \ref{th:Bel9}, any of the topological spaces is compact iff $M$ has a top element.
\end{proof}

\begin{thm}\label{th:Bel11}
Let $M$ be an $EMV$-algebra. Then the topological spaces $\mathcal{SM}(M)$ and $\mathrm{MaxI}(M)$ are locally compact Hausdorff space such that if $a$ is an idempotent, then $S(a)$ and $M(a)$ are compact clopen subsets. If $M$ has a top element, then $\mathcal{SM}(M)$ and $\mathrm{MaxI}(M)$ are compact spaces.
\end{thm}

\begin{proof}
Due to Basic Representation Theorem \ref{th:embed}, either $M$ has a top element, and $M$ is an $MV$-algebra, or $M$ can be embedded into $N$ as its maximal ideal, and every $x\in N$ either belongs to $M$ or $\lambda_1(x)$ belongs to $M$. If $M$ has a top element, then $\mathcal{SM}(M)$ and $\mathrm{MaxI}(M)$ are compact and homeomorphic, see Theorem \ref{th:Bel9}.

Let us assume $M$ has no top element. Given $x \in M$ and $y\in N$, let $S(x)=\{s\in\mathcal{SM}(M)\mid s(x)>0\}$ and $S_N(y)= \{s\in\mathcal{SM}(N) \mid s(y)>0\}$, they are open sets.

Define a mapping $\phi: \mathcal{SM}(M) \to \mathcal{SM}(N)$ by $\phi(s)=\tilde s$, $s \in \mathcal{SM}(M)$, where $\tilde s$ is defined by (\ref{eq:state}). Then $\phi$ is an injective mapping such that $\phi(S(x))= S_N(x)$ for each $x \in M$.
Take an idempotent $a \in \mathcal I(M)$. Then $S(a)=\{s\in \mathcal{SM}(M) \mid s(a)>0\}=\{s \in \mathcal{SM}(M) \mid s(a)=1\}$ is both open and closed. The same is true for $S_N(a)=\{s\in \mathcal{SM}(N)\mid s(a)>0\}$, in addition $S_N(a)$ is compact because $\mathcal{SM}(N)$ is compact.

For each $x\in M$ and $u,v$ real numbers with $u<v$, the sets $S(x)_{u,v}=\{s\in \mathcal{SM}(M)\mid u<s(x)<v\}$ and $S_N(x)_{u,v}=\{s\in \mathcal{SM}(N)\mid u<s(x)<v\}$, where $x \in N$, are open and they form a subbase of the weak topologies. Then $\phi(S(x)_{u,v})= S_N(x)_{u,v}$ and $\phi(S(x))=S_N(x)$ whenever $x \in M$.

Now we show that $S(a)$ is a compact set in $\mathcal{SM}(M)$. Take an open cover of $S(a)$ in the form $\{S(x_\alpha)_{u_\alpha,v_\alpha}\mid \alpha \in A\}$, where $x_\alpha \in M$ and $u_\alpha, v_\alpha$ are real numbers such that $u_\alpha< v_\alpha$ for each $\alpha \in A$. Then
\begin{eqnarray*}
S(a) &\subseteq & \bigcup_\alpha S(x_\alpha)_{u_\alpha,v_\alpha}\\
\phi(S(a))&\subseteq&  \bigcup_\alpha \phi(S(x_\alpha)_{u_\alpha,v_\alpha})\\
S_N(a)&\subseteq& \bigcup_\alpha \phi(S(x_\alpha)_{u_\alpha,v_\alpha}).
\end{eqnarray*}
The compactness of $S_N(a)$ entails a finite subset $F$ of $A$ such that $S_N(a)\subseteq \bigcup\{ \phi(S(x_\alpha)_{u_\alpha,v_\alpha})\mid \alpha \in F\}$, whence, $S(a)\subseteq \bigcup\{ S(x_\alpha)_{u_\alpha,v_\alpha}\mid \alpha \in F\}$. Since the system of all open sets $S(x)_{u,v}$ forms a subbase of the weak topology of $\mathcal{SM}(M)$, we have by \cite[Thm 5.6]{Kel}, $S(a)$ is compact and clopen as well. In addition, given a state-morphism $s \in \mathcal{SM}(M)$, there is an element $x\in M$ with $s(x)=1$, and there is an idempotent $a\in M$ such that $x\le a$ which entails $s \in S(x)\subseteq S(a)$. Whence, $\mathcal{SM}(M)$ is locally compact.

\vspace{2mm}
\noindent {\it Claim.  $M(a)$ and $M_N(a)$ are both clopen and compact.}
\vspace{1mm}

Define a mapping $\theta_N:\mathcal{SM}(N) \to \mathrm{MaxI}(N)$ by $\theta_N(s):=\Ker(s)$, $s \in \mathcal{SM}(N)$. Since $N$ has a top element, $\theta_N$ is a homeomorphism, see Theorem \ref{th:Bel9}. Therefore, $M_N(a)$ is clopen and compact.

Whence $M_N(a)$ is compact in $\mathrm{MaxI}(N)$. We show that also $M(a)$ is compact in $\mathrm{MaxI}(M)$. Take an open covering $\{M(x_\alpha)\mid \alpha \in A\}$ of $M(a)$, where each $x_\alpha \in M$. Given $I\in \mathrm{MaxI}(M)$, there is a unique state-morphism $s$ on $M$ such that $I=\Ker(s)=\theta^{-1}(s)$, therefore, we define the mapping $\psi: \mathrm{MaxI}(M)\to \mathrm{MaxI}(N)$ by $\psi(I)=\theta_N^{-1}(\tilde s)$.

Then $\{\psi(M(x_\alpha))\mid \alpha \in A\}$ is an open covering of $\psi(M(a))=M_N(a)$ which is a compact set. Whence, there is a finite subcovering $\{\psi(M(x_{\alpha_i}))\mid i=1,\ldots,n\}$ of $\psi(M(a))$, consequently $\{M(x_{\alpha_i})\mid i=1,\ldots,n\}$ is a finite subcovering of $M(a)$, consequently, $M(a)$ is compact and clopen as well.
\end{proof}

\begin{cor}\label{co:Bel12}
Let $M$ be an $EMV$-algebra with the general comparability property. Then
the spaces $\mathcal{SM}(\mathcal I(M))$, $\mathcal{SM}(M)$, $\mathrm{MaxI}(\mathcal I(M))$, and $\mathrm{MaxI}(\mathcal I(M))$ are totally disconnected, locally compact and completely regular spaces.
\end{cor}

\begin{proof}
By Theorem \ref{th:Bel10}, all spaces are mutually homeomorphic, and by Theorem \ref{th:Bel11}, they are completely regular, locally compact and totally disconnected.
\end{proof}

We say that a topological space $\Omega$ is {\it Baire} if, for each sequence of open and dense subsets $\{U_n\}$, their intersection $\bigcap_n U_n$ is dense.

\begin{cor}\label{co:Baire}
Let $M$ be an $EMV$-algebra. The spaces $\mathcal{SM}(M)$ and $\mathrm{MaxI}(M)$ are Baire spaces.
\end{cor}

\begin{proof}
Both spaces are homeomorphic, see Theorem \ref{th:Bel9}, due to Theorem \ref{th:Bel11}, both spaces are locally compact, and by Proposition \ref{pr:Bel8}, they are completely regular. Therefore, they are also regular. Applying the Baire Theorem, \cite[Thm 6.34]{Kel}, the spaces are Baire spaces.
\end{proof}

Motivated by Example \ref{ex:LS10}, we have the following result which describes the state-morphisms spaces of $M$ and $N$ from the topological point of view.

\begin{thm}\label{th:comp}
Let $M$ be an $EMV$-algebra without top element which is a maximal ideal of the $MV$-algebra $N=\{x\in N\mid \text{ either } x\in M \text{ or } \lambda_1(x)\in M\}$. Then $\mathcal{SM}(N)$ and $\mathrm{MaxI}(N)$ are the one-point compactifications of the spaces $\mathcal{SM}(M)$ and $\mathrm{MaxI}(M)$, respectively.
\end{thm}

\begin{proof}
In what follows, we use the result and notation from Proposition \ref{pr:corr1}. By Theorem \ref{th:Bel9}, $\mathcal{SM}(N)$ is a compact Hausdorff topological space, whereas $\mathcal{SM}(M)$ is, according to Theorem \ref{th:Bel11}, a locally compact Hausdorff topological space. Due to the Alexander theorem, see \cite[Thm 4.21]{Kel}, there is the one-point compactification of $\mathcal{SM}(M)$. We are going to show that the one-point compactification of $\mathcal{SM}(M)$ is $\mathcal{SM}(N)$.

We proceed in five steps.

(1) If $O_N$ is an open set of $\mathcal{SM}(N)$ such that $s_\infty\notin O_N$, then $O_N=\phi(O)$ for some open subset $O$ of $\mathcal{SM}(M)$. %For each $\tilde s \in O_N$, there are $x_1,\ldots,x_n \in M$ and real numbers $0\le u_i< v_i\le 1$ for $i=1,\ldots,n$ such that $\tilde s\in \bigcap_{i=1}^n S_N(x_i)_{u_i,v_i}\subseteq O$. Hence, $s\in \bigcap_{i=1}^n S(x_i)_{u_i,v_i}$ which means that $O$ is an open subset of $\mathcal{SM}(M)$.

(2) Now take an open set $O_N$ containing $s_\infty$ and $O_N=S_N(x)_{u,v}$,  where $x \in  M$ and $u,v$ are real numbers with $ u< v$. Since $s_\infty(x)=0$, $u<0<v$ and we have $S_N(x)_{u,v}=\{s_\infty\}\cup\{\tilde s\mid s\in \mathcal{SM}(M), s(x)<v\}=\{s_\infty\}\cup \phi(\{s\in \mathcal{SM}(N) \mid s(x)<v\})$. If $X:=\phi(\mathcal{SM}(M))\setminus (\{s_\infty\}\cup \phi(\{s\in \mathcal{SM}(N) \mid s(x)<v\})$, then $X=\{s\in \mathcal{SM}(M) \mid s(x)\ge v\}\subseteq \{s\in \mathcal{SM}(N) \mid s(a)\ge v\}$, where $a\in \mathcal I(M)$ such that $x\le a$. If $u \ge 1$, then $X=\emptyset$ which is a compact set and if $u<1$, then $X\subseteq \{s\in \mathcal{SM}(M) \mid s(a)=1\}$. Since the latter set is compact, see Theorem \ref{th:Bel11}, we see that $X$ is closed, and consequently, $X$ is compact, too.

(3) Now let $s_\infty\in O_N=S_N(x)_{u,v}$, where $x \in  M$ and $u,v$ are real numbers with $ u< v$ and $x=\lambda_1(x_0)$, where $x_0\in M$. Since $s_\infty(x)=1$, we have $v>1$. Then $S_N(x)_{u,v}=\{s_\infty\} \cup \{\tilde s \mid s \in \mathcal{SM}(M), u <\tilde s(x)\}=\{s_\infty\}\cup \phi(\{ s\in \mathcal{SM}(M) \mid s(x_0)<1-u\})$. Therefore, $\phi(\mathcal{SM}(M))\setminus(\{s_\infty\}\cup \phi(\{s\in \mathcal{SM}(M)\mid s(x_0)<1-u\})=\phi(\mathcal{SM}(M) \setminus \{s\in \mathcal{SM}(M)\mid s(x_0)<1-u\})=\phi(\{s\in \mathcal{SM}(M)\mid s(x_0)\ge 1-u\})$ and $X=\{s\in \mathcal{SM}(M)\mid s(x_0)\ge 1-u\}=\emptyset$, which is a compact set, if $u<0$, and $X\subseteq \{s\in \mathcal{SM}(M)\mid s(a)\ge 1-u\}= \{s\in \mathcal{SM}(M)\mid s(a)=1\}$ if $u\ge 0$ and $a$ is an idempotent of $M$ with $x_0\le a$. Therefore, $X$ is a closed subset which is a subset of a compact set, see Theorem \ref{th:Bel11}, and we have $X$ is a compact set.

(4) Let $s_\infty\in O_N = \bigcap_{i=1}^n S_N(x_i)_{u_i,v_i}$, where $u_i \in N$, $u_i<v_i$  and $s_\infty \in S_N(x_i)_{u_i,v_i}$ for each $i=1,\ldots,n$. Then $S_N(x_i)_{u_i,v_i}=\{s_\infty\}\cup \phi(S(x'_i)_{u'_i,v'_i})$ where if $x_i \in M$, then $x'_i=x_i$ and $u'_i=u_i$, $v'_i=v_i$ and if $x_i\in N\setminus M$, then $x'_i=\lambda_1(x_i)$ and $u'_i=1-v_i$, $v'_i=1-u_i$.

Hence, $\phi(\mathcal{SM}(M))\setminus \bigcap_{i=1}^n S_N(x_i)_{u_i,v_i}=\phi(\mathcal{SM}(M)\setminus(\{s_\infty\} \cup \phi(\bigcap_{i=1}^n S(x'_i)_{u'_i,v'_i})))=\phi(\bigcup_{i=1}^n (\mathcal{SM}(M) \setminus S(x'_i)_{u'_i,v'_i}))$, so that $\bigcup_{i=1}^n (\mathcal{SM}(M) \setminus S(x'_i)_{u'_i,v'_i})$ is a compact set in view of (3).

(5) $O_N=\bigcup_\alpha O^N_\alpha$, where each $O^N_\alpha$ is the set of the form (4). Then $O^N_\alpha=\{s_\infty\}\cup \phi(O_\alpha)$ if $s_\infty \in O^N_\alpha$, otherwise $O^N_\alpha = O_\alpha$, where $O_\alpha $ is an open set in $\mathcal{SM}(M)$.

Then $\phi(\mathcal{SM}(M) \setminus \bigcup_\alpha O^N_\alpha)= \phi(\mathcal{SM}(M)\setminus\bigcup_\alpha O_\alpha)$, where $O_\alpha$ is a subset of $\mathcal{SM}(M)$ such that $O^N_\alpha =\phi(O_\alpha)$. Whence, $\mathcal{SM}(M) \setminus \bigcup_\alpha O_\alpha = \bigcap_\alpha(\mathcal{SM}(M) \setminus O_\alpha) \subseteq \mathcal{SM}(M) \setminus O_{\alpha_0}$, where $\alpha_0$ is an index $\alpha$ such that $s_\infty \in O^N_{\alpha_0}$, which is by (4) a compact set, consequently, $\bigcap_\alpha(\mathcal{SM}(M) \setminus O_\alpha)$ is a compact set.

Therefore, $\mathcal{SM}(N)$ is the one-point compactification of $\mathcal{SM}(M)$.

Since the spaces $\mathcal{SM}(M)$ and $\mathrm{MaxI}(M)$ are homeomorphic, see Theorem \ref{th:Bel9}, the same is true for $\mathcal{SM}(N)$ and $\mathrm{MaxI}(N)$. If we define $I_\infty =M$, $I_\infty$ is a maximal ideal of $N$, and $I_\infty=\Ker(s_\infty)$. In addition, if $s \in \mathcal{SM}(M)$, then $\Ker(\tilde s)\cap M=\Ker(s)$.
Therefore, we get that the one-point compactification of $\mathrm{MaxI}(M)$ is $\mathrm{MaxI}(N)=\{\Ker(\tilde s)\mid s \in \mathcal{SM}(M)\}\cup\{I_\infty\}$.
\end{proof}

In a dual way as we did for the set of maximal ideals, we define the hull-kernel topology on the set $\mathrm{MaxF}(M)$ of maximal filters on an $EMV$-algebras $M$.  Thus given a filter $F$ from the set $\mathrm{Fil}(M)$ of all filters on $M$, we define
$$O_1(F):=\{B \in \mathrm{MaxF}(M)\mid F \varsubsetneq B\}.$$
Then (i) $F_1 \subseteq F_2$ implies $O_1(F_1)\subseteq O_1(F_2)$, (ii)
$\bigvee_\alpha O_1(F_\alpha)=O_1(\bigvee_\alpha F_\alpha)$, (iii) $\bigcup\{O_1(F)\mid F \in \mathrm{Fil}(M)\}=O_1(M)=\mathrm{Fil}(M)$, (iv) $\bigcap_{i=1}^n O_1(F_i)=O_1(\bigcap_{i=1}^n F_i)$. Hence, the system $\{O_1(F)\mid F \in \mathrm{Fil}(M)\}$ defines the so-called hull-kernel topology on the set $\mathrm{MaxF}(M)$. Every closed set is of the form $C_1(F)=\{B \in \mathrm{MaxF}(M)\mid F \subseteq B\}$. If given $x \in M$, we set $M_1(x)=\{B \in \mathrm{MaxF}(M) \mid x\notin B\}$, then the system $\{M_1(x)\mid x \in M\}$ is a base for the hull-kern topology of maximal filters.

The following result is dual to the one from Proposition \ref{pr:Bel8}.

\begin{prop}\label{pr:corr2}
Let $X$ be a non-empty set of state-morphisms closed in the weak topology of state-morphisms of an $EMV$-algebra $M$. Let $t$ be a state-morphism such that $t\notin X$. There exists an element $a\in M$ such that $t(a)=0$ and $s(a)=1$.
\end{prop}

\begin{proof}
Since the proof of the statement is dually similar to the one of Proposition \ref{pr:Bel8}, we outline only the main steps.

Let $t$ be a state-morphism such that $t \notin X$.
We assert that there exists an $a \in M$ such
that $t(a) < 1/2$ while $s(a) > 1/2$ for all $s \in X.$

Indeed, set $A= \{a\in M:\ t(a) <1/2\}$, and for all $a \in A$,
let
$$W(a) :=\{s\in \mathcal{SM}(M))\mid\ s(a)> 1/2\},
$$
which is an open subset of $\mathcal{SM}(M)$. We note that $A\ne \emptyset$ and $A$ is upward directed and closed under $\odot$.

We assert that these open subsets cover $X$. Consider any $s \in
X$. Since $\Ker(s)$ and $\Ker(t)$ are non-comparable subsets of
$M$, there exists $x \in \Ker(t)\setminus \Ker(s).$ Hence $t(x) =
0$ and $s(x)>0$. There exists an integer $n \ge 1$ such that
$s(n.x) > 1/2$. Then $t(n.x)=0$. If we put $a=n.x$, then $s \in W(a)$.
Therefore, $\{W(a)\mid a \in A\}$ is an open covering of $X$.

Similarly as in the proof of Proposition \ref{pr:Bel8}, we pass to $\mathcal{SM}(N)$, where $N$ is an $MV$-algebra such that $M$ is an $EMV$-subalgebra of $N$ and we take the compact space $\phi(X)\cup\{s_\infty\}$.  For each $a\in A$, we define $\widetilde W(a)=\{s\in \mathcal{SM}(N)\mid s(a)>1/2\}$. Then each $\widetilde W(a)$ is open subset of $\mathcal{SM}(N)$ not containing $s_\infty$. Therefore, let $b\in M$ be arbitrary element and we set $\widetilde W(b)=\{s\in \mathcal{SM}(N) \mid s(b)<1/2\}$. Then $\widetilde W(b)$ is an open set containing the state-morphism $s_\infty$, and $\widetilde{\mathcal W(b)}$ is disjoint with $\widetilde W(a)$ for each $a \in A$. Since $\{\widetilde W(a)\mid a \in A\}\cup \{\widetilde W(b)\}$ is an open covering of $\phi(X)\cup\{s_\infty\}$, so that there are $a_1,\ldots,a_n\in A$ such that $\phi(X)\cup\{s_\infty\} \subseteq \bigcup_{i=1}^n \widetilde W(a_i) \cup \widetilde W(b)$. Therefore $X \subseteq W(a_1) \cup \cdots \cup W(a_n)$ for
some $a_1, \ldots, a_n \in A$. Put $a_0 = a_1 \vee \cdots \vee
a_n$. Then $a_0 \in A$ and for each $s \in X$, we have $s(a_0) \ge s(a_i) > 1/2$ for $i =1,\ldots, n, $ which proves $X \subseteq W(a_0)$, i.e., $s(a_0) >  1/2$ for all $s \in X$. If we put $a = a_0\oplus a_0$, then $t(a)=0$ and $s(a)=1$ for each $s \in X$.
\end{proof}

\begin{thm}\label{th:Fil}
Let $M$ be an $EMV$-algebra. Then the spaces $\mathcal{SM}(M)$, $\mathrm{MaxI}(M)$ and $\mathrm{MaxF}(M)$ are mutually homeomorphic spaces.
\end{thm}

\begin{proof}
According to Theorem \ref{th:Bel9}, the spaces $\mathcal{SM}(M)$ and $\mathrm{MaxI}(M)$ are homeomorphic and the mapping $\theta:\mathcal{SM}(M) \to \mathrm{MaxI}(M)$, defined by $\theta(s)=\Ker(s)$, is a homeomorphism. According to Lemma \ref{le:Bel5}, the mapping $\zeta: \mathcal{SM}(M)\to \mathrm{MaxF}(M)$ given by $\zeta(s)=\Ker_1(s)$, $s \in \mathcal{SM}(M)$, is bijective.

Let $C_1(F)$ be any closed subspace of $\mathrm{MaxF}(M)$. Then
$$\theta^{-1}(C_1(F))=\{s \in \mathcal{SM}(M) \mid s(x)=1 \text{ for all } x \in F\}
$$
is a closed subspace of $\mathcal{SM}(M)$, so that $\zeta$ is continuous.

Given a non-empty subset $X$ of $\mathcal{SM}(M)$, we define
$$\Ker_1(X):=\{x \in M \mid s(x)=1 \text{ for all } s \in X\}.
$$
Then $\Ker_1(X)$ is a filter of $M$. If, in addition, $X$ is a closed subset of $\mathcal{SM}(M)$, we assert
$$\zeta(X)=C_1(\Ker_1(X)).$$

The inclusion $\zeta(X)\subseteq C_1(\Ker_1(X))$ is evident. By Proposition \ref{pr:corr2}, if $t \notin X$, there is an element $a\in M$ such that $s(a)=1$ for each $s\in X$ and $t(a)=0$. Consequently, $t \notin X$ implies $\zeta(t) \notin C(\Ker_1(X))$, and $C(\Ker_1(X)) \subseteq \zeta(X)$.
As a result, we conclude $\zeta$ is a homeomorphism.
\end{proof}

\begin{lem}\label{le:Bel13}
Let $M$ be an $EMV$-algebra, $x\in M$, and $b \in \mathcal I(M)$ with $x\le b$.

{\rm (i)} Then
$$M(b)\setminus M(x)\subseteq M(\lambda_b(x)).
$$

{\rm (ii)} If $x\in \mathcal I(M)$, then
$$M(b)\setminus M(x)= M(\lambda_b(x)).
$$

{\rm (iii)} If $x,y \in M$, $x,y \le b\in \mathcal I(M)$, then
$$ M(y)\setminus M(x\wedge y)= M(y)\setminus M(x) \subseteq M(y\odot \lambda_b(x))\subseteq M(\lambda_b(x)).
$$

{\rm (iv)} Let $M$ be semisimple, $x \in M$, and $x \le b \in \mathcal I(M)$. Then $x\in M$ is an idempotent if and only if $M(b)\setminus M(x)= M(\lambda_b(x))$.

{\rm (v)} Let $M$ be semisimple, $x,y \in \mathcal I(M)$, and $x,y \le b \in \mathcal I(M)$. Then
$$ M(y)\setminus M(x\wedge y)= M(y)\setminus M(x) = M(y\odot \lambda_b(x)).
$$

{\rm (vi)} If $M$ is an arbitrary $EMV$-algebra having a top element $1$, then for each idempotent $a \in \mathcal I(M)$, we have $M(\lambda_1(a))=M(1)\setminus M(a)= M(a)^c$, where $M(a)^c$ is the set complement of $M(a)$ in $\mathrm{MaxI}(M)$.
\end{lem}

\begin{proof}
(i) Let $x\le b\in \mathcal I(M)$ and take $A \in M(b)\setminus M(x)$. Then $b \notin A$ and $x\in A$. We assert $\lambda_b(x)\notin A$. If not then from $b = x \oplus \lambda_b(x)$ we get a contradiction.

(ii) Assume that also $x$ is an idempotent and take $A \in M(\lambda_b(x))$. Due to $b = x\oplus \lambda_b(x)$, we have $\lambda_b(x)\notin A$ and $b \notin A$. Since $A$ is a prime ideal of $M$, then $0 =\lambda_a(x)\odot x= \lambda_b(x)\wedge x \in A$ entails $x \in  A$ so that $A \in  M(b) \setminus M(x)$.

(iii) Let $x,y \le b \in \mathcal I(M)$. We have $M(y)\setminus M(x\wedge y)
= M(x)\setminus (M(x)\wedge M(y))=M(x)\setminus M(y)$. Choose $A \in M(x)\setminus M(y)$. Then $x \notin A$ and $y\in A$. Due to (\ref{eq:x<y2}), we have $y = (x\wedge y)\oplus (y \odot \lambda_b(x))$ so that we get $y \odot \lambda_b(x)\notin A$. It is evident that $M(y \odot \lambda_b(x))\subseteq M(\lambda_b(x))$.

(iv) Now let $M$ be semisimple and $x \le b \in \mathcal I(M)$. If $x$ is idempotent, we have already established in (ii) $M(b)\setminus M(x)= M(\lambda_b(x))$. Conversely, let $M(b)\setminus M(x)= M(\lambda_b(x))$. Then for each $A \in M(b)$, we have either $x \in A$ or $\lambda_b(x)\notin A$. Whence $x\wedge \lambda_b(x)\in A$, and since $A\cap [0,b]$ is a maximal ideal of the $MV$-algebra $[0,b]$, \cite[Prop 3.23]{DvZa}, we have $x\wedge \lambda_b(x)\in [0,b]\cap A$; the same is true if $A\notin M(b)$, whence it holds for each maximal ideal $A$ of $M$. Since $M$ is semisimple, $x\wedge \lambda_b(x)=0$ and $x$ is an idempotent in the $MV$-algebra $[0,b]$, so it is an idempotent in $M$, too.

(v) Let $A \in M(y\odot \lambda_b(x))$. Then $y\odot\lambda_b(x)\notin A$ and $y,\lambda_b(x) \notin A$. Due to  (\ref{eq:x<y2}), we have $y = (x\wedge y)\oplus (y\odot\lambda_b(x))$ and $(x\wedge y)\wedge (y\odot\lambda_b(x))= (x\odot y)\odot (y\odot\lambda_b(x))=0\in A$ ($x,y$ and also $\lambda_b(x)$ are idempotents). Then $x\wedge y \in A$ and in addition, $x \in A$. Therefore, $A \in M(y)\setminus M(x)$.

(vi) If $1$ is a top element of $M$, $a \in \mathcal I(M)$, then the assertion follows from the above proved equality.
\end{proof}

\begin{prop}\label{pr:Bel15}
Let $M$ be a semisimple $EMV$-algebra. If $x = \bigvee_t x_t\in M$, then
$$
M(x) \setminus \bigcup_t M(x_t)
$$
is a nowhere dense subset of  $\mathrm{MaxI}(M)$.
\end{prop}

\begin{proof}
Let $x = \bigvee_t x_t $ and
suppose $M(x) \setminus \bigcup_t M(x_t)$ is not nowhere dense.
Since $\{M(y)\mid y \in M\}$ is
a base of the  topological space $\mathcal T_M$, there exists a non-zero
element $b \in M$ such that $\emptyset \ne M(b) \subseteq M(x) \setminus \bigcup_t M(x_t)$.
Due to $M(b) = M(b) \cap M(x) = M(b \wedge x)$, we take $b_0 :=
b \wedge x$ which is a non-zero element of $M$. Then $M(b_0)
\cap M(x_t) = \emptyset$ for any $t$, so that $M(b_0 \wedge x_t) =
\emptyset$ and the semisimplicity of $M$ yields $b_0 \wedge x_t=0$
for any $t$.

Using Proposition \ref{pr:Bel4}, we have
$$
b_0 = b_0 \wedge a= b_0 \wedge \bigvee_tx_t = \bigvee_t(b_0 \wedge x_t)
= 0,
$$
which gives $M(b) = \emptyset,$
 a contradiction, so that our assumption was false, and consequently,
$M(x) \setminus \bigcup_t M(x_t)$ is a nowhere dense set.
\end{proof}

\begin{prop}\label{pr:Bel16}
Let $M$ be a semisimple $EMV$-algebra and let $x_t \le x\le a \in \mathcal I(M)$ for any $t$. If $\bigcap_t M(x\odot \lambda_a(x_t))$ is a nowhere
dense subset of $\mathrm{MaxI}(M)$, then $x = \bigvee_t x_t$.
\end{prop}

\begin{proof}
It is clear that in order to prove $x = \bigvee_t x_t$
it is sufficient to verify that $x_t \le y \le x$ for any $t$
implies $y = x$.

So let $\bigcap_t M(x\odot \lambda_a(x_t))$ be a nowhere dense set, and let $y \ne x$ for some $y \ge x_t$, $y \le x$. Then $x \odot \lambda_a(y) \ne 0$ and $M(x \odot \lambda_a(y))$ is a non-empty open subset of $\mathrm{MaxI}(M)$. By assumptions, there exists a non-zero open subset $O \subseteq M(x\odot \lambda_a(y))$ such that $O \cap \bigcap_t M(x\odot \lambda_a(x_t)) = \emptyset$. Consequently,
there is a non-zero element $z \in M$ such that $M(z) \subseteq
O$. Hence, for any $A \in M(z) \subseteq M(x \odot \lambda_a(y))$, we have
$z \notin A$, $x \odot \lambda_a(y) \notin A$ and $A \notin \bigcap_t M(x\odot \lambda_a(x_t))$.
This entails that there is an index $t$ such that $x \odot \lambda_a(x_t)
\in A$. Since $x_t \le y$, we have $x \odot \lambda_a(y) \le x \odot \lambda_a(x_t) \in
A$ which implies $x \odot \lambda_a(y) \in A$, and this is a contradiction
with $x \odot \lambda_a(y) \notin A$. Finally, our assumption $y < x$ was
false, and whence $y = x$ and $x = \bigvee_t x_t$.
\end{proof}

\begin{cor}\label{co:Bel17}
Let $M$ be a generalized Boolean algebra. Let $\{x_t\}$ be a system of elements of $M$ which is majorized by $x\in M$. Then $x=\bigvee_t x_t$ if and only if $M(x)\setminus \bigcup_t M(x_t)$ is a nowhere dense set of $\mathrm{MaxI}(M)$.
\end{cor}

\begin{proof}
By \cite[Lem 4.8]{DvZa}, $M$ is a semisimple $EMV$-algebra.
If $x = \bigvee_t x_t$, the statement follows from Proposition \ref{pr:Bel15}.  Conversely, let $M(x)\setminus \bigcup_t M(x_t)$ be a nowhere dense. Then by Lemma \ref{le:Bel13}(v), we have $M(\lambda_x(x_t))=M(x\wedge \lambda_x(x_t))= M(x \odot \lambda_x(x_t))=M(x)\setminus M(x_t)$, so that $\bigcup_t M(\lambda_x(x_t))=M(x)\setminus \bigcap_t M(x_t)$ is a nowhere dense set and applying Proposition \ref{pr:Bel16}, $x= \bigvee_t x_t$.
\end{proof}

\begin{cor}\label{co:Bel18}
A generalized Boolean algebra $M$ is Dedekind $\sigma$-complete if and only if, for each sequence $\{a_n\}$ of elements of $M$ which is majorized by an element $a\in M$, we have $\bigvee_n a_n = a$ if and only if $M(a)\setminus \bigcup_n M(a_n)$ is a nowhere dense set of $\mathrm{MaxI}(M)$.
\end{cor}

\begin{proof}
It follows from Corollary \ref{co:Bel17}.
\end{proof}

\begin{prop}\label{pr:Bel19}
Let $M$ be an $EMV$-algebra. For each $x \in M$, we have
\begin{equation}\label{eq:nx}
M(x)= \bigcup_{n=1}^\infty\Big(M(a)\setminus M(\lambda_a(n.x))\Big),
\end{equation}
where $a$ is an idempotent of $M$ such that $x\le a$.
\end{prop}

\begin{proof}
If $x\in \Rad(M)$, then $M(x)=\emptyset$. If $a \in \Rad(M)$, then $M(a)=M(\lambda_a(n.x))=\emptyset$ and (\ref{eq:nx}) holds. If $a\notin \Rad(M)$, then $M(a)\ne \emptyset$.  From $a = n.x \oplus \lambda_a(n.x)$ we conclude $A\in M(a)$ iff $\lambda_a(n.x)\notin A$, so that $M(a)=M(\lambda_a(n.x))$ for each $n\ge 1$, henceforth (\ref{eq:nx}) holds.

Now let $x \notin \Rad(M)$. Then $M(x)\ne \emptyset$ and let $A\in M(x)$. Again from $a = n.x \oplus \lambda_a(n.x)$, we conclude $A \notin M(a)$ and there is an integer $n\ge 1$ such that $\lambda_a(n.x)\in A$. Therefore, $M(x) \subseteq \bigcup_{n=1}^\infty(M(a)\setminus M(\lambda_a(n.x)))$.

Now, if $\bigcup_{n=1}^\infty(M(a)\setminus M(\lambda_a(n.x)))$ is empty, then $M(x)=\emptyset$ and the equality holds. Thus let $A \in \bigcup_{n=1}^\infty(M(a)\setminus M(\lambda_a(n.x)))$. There is an integer $n\ge 1$ such that $A \in M(a)\setminus M(\lambda_a(n.x))$ which means $a\notin A$ and $\lambda_a(n.x)\in A$. From $a = n.x \oplus \lambda_a(n.x)$, we have $n.x \notin A$, so that $x\notin A$ and $A \in M(x)$ which proves (\ref{eq:nx}).
\end{proof}

\section{Loomis--Sikorski Theorem for $\sigma$-complete $EMV$-algebras}%5

In this section, we define a stronger notion of $\sigma$-complete $EMV$-algebras than Dedekind complete $EMV$-algebras and for them we establish a variant of the Loomis--Sikorski theorem which will say that every $\sigma$-complete $EMV$-algebra is a $\sigma$-homomorphic image of some $\sigma$-complete $EMV$-tribe of fuzzy sets, where all operations are defined by points.

We say that an $EMV$-algebra $M$ is $\sigma$-{\it complete} if any countable family  $\{x_n\}$ of elements of $M$ has the least upper bound in $M$. Clearly, every $\sigma$-complete $EMV$-algebra is Dedekind $\sigma$-complete. Therefore, all results of the previous section concerning Dedekind $\sigma$-complete $EMV$-algebras are valid also for $\sigma$-complete ones. We note that both notions coincide if $M$ has a top element. In opposite case these notions may be different. Indeed, let $\mathcal T$ be the set of all finite subsets of the set $\mathbb N$ of natural numbers. Then $\mathcal T$ is a generalized Boolean algebra that is Dedekind $\sigma$-complete but not $\sigma$-complete. On the other hand, if $\mathcal T$ is a system of all finite or countable subsets of the set of reals, then $\mathcal T$ is a $\sigma$-complete generalized Boolean algebra without top element.

\begin{lem}\label{le:LS5}
Let $M$ be a $\sigma$-complete $EMV$-algebra. Then no non-empty open set of $\mathcal{SM}(M)$ can be expressed as a countable union of nowhere dense sets.
\end{lem}

\begin{proof}
By Proposition \ref{pr:Bel3}, $M$ satisfies the general comparability property, and by Theorem \ref{th:Bel10}, the spaces $\mathcal{SM}(M)$, $\mathrm{MaxI}(M)$, $\mathcal{SM}(\mathcal I(M))$ and $\mathrm{MaxI}(\mathcal I (M))$ are mutually homeomorphic spaces. In addition, $\mathcal I(M)$ is $\sigma$-complete. Therefore, we prove lemma for $\mathrm{MaxI}(\mathcal I (M))$. We note, that given $x \in \mathcal I(M)$, $M(x)=\{I \in \mathrm{MaxI}(\mathcal I (M))\mid x\notin I\}$, and by Theorem \ref{th:Bel11}, $M(x)$ is clopen and compact.

Let $O\ne\emptyset$ be an open set of $\mathrm{MaxI}(\mathcal I (M))$ and let $O=\bigcup_n S_n$, where each $S_n$ is a nowhere dense subset of $\mathrm{MaxI}(\mathcal I (M))$. Let $O_0$ be a non-empty open set, there is $x_1\ne 0$ such that $M(x_1)\subseteq O_0$ and $M(x_1)\cap S_1=\emptyset$. Since also $S_2$ is nowhere dense, in the same way, there is $0<x_2\in M$ such that $M(x_2)\subseteq M(x_1)$ and $M(x_2)\cap S_2=\emptyset$. By induction, we obtain a sequence of non-zero elements $\{x_n\}$ such that $M(x_{n+1})\subseteq M(x_n)$ and $M(x_n)\cap S_n=\emptyset$. We define $y_n=x_1\wedge\cdots\wedge x_n$ for each $n\ge 1$. Then $M(y_n)=M(x_n)$, $n\ge 1$, and $M(y_n)\subseteq M(y_1)$. Put $y_0 = \bigwedge_n y_n$. Since $M(y_1)$ is compact, $\bigcap_n M(y_n)\ne \emptyset$, otherwise there is an integer $n_0$ such that $M(y_{n_0})= \bigcap_{i=1}^{n_0}M(y_i) = \emptyset$, a contradiction.

Therefore, there is a maximal ideal $I$ belonging to each $M(y_n)$ and $I \notin S_n$, so that $I \notin \bigcup_n S_n$ which is absurd, and the lemma is proved.
\end{proof}

Given an element $x \in M$, the set $S(x)$ was defined  as $S(x)=\{s \in \mathcal{SM}(M) \mid s(x)>0\}$.

\begin{thm}\label{th:LS2}
Let $M$ be a $\sigma$-complete $EMV$-algebra. For each $x \in M$, we define
\begin{equation}\label{eq:a(x)}
a_0(x):=\bigvee_n n.x.
\end{equation}
Then $a_0(x)$ is an idempotent of $M$ such that $a_0(x)\ge x$ and
\begin{equation}\label{eq:LSa}
a_0(x)=\bigwedge\{a\in \mathcal I(M)\mid a\ge x\}.
\end{equation}
In addition, $\overline{S(x)}=S(a_0(x))$, and if $\overline{S(x)}=S(b)$ for some idempotent $b\in \mathcal I(M)$, then $a_0(x)=b$.

On the other hand, there is an idempotent $b_0(x)$ of $M$ such that
$$b_0(x)=\bigwedge_n x^n
$$
and
\begin{equation}\label{eq:LSb}
b_0(x)= \bigvee \{b\in \mathcal I(M) \mid b\le x\}.
\end{equation}

{\rm (1)} If $y$ is an element of $M$ such that $x\le y$ and if $b$ is an idempotent with $\overline{S(y)}=S(b)$, then $a_0(x)\le b$.

{\rm (2)} Let $x, x_1,\ldots $ and $a,a_1,\ldots$ be a sequence of elements of $M$ and $\mathcal I(M)$, respectively, such that $\overline{S(x)}=S(a)$ and $\overline{S(x_n)}= S(a_n)$ for each $n\ge 1$. If $x=\bigvee_n x_n$, then $a =\bigvee_n a_n$.
\end{thm}

\begin{proof}
Since $M$ is $\sigma$-complete, the element $a_0(x)=\bigvee_n n.x$ exists in $M$ for each $x \in M$. Using \cite[Prop 1.21]{georgescu}, we have $a_0(x)\oplus a_0(x)= a_0(x)\oplus \bigvee_n n.x=\bigvee_n(a_0(x)\oplus n.x)=\bigvee_n\bigvee_m (n+m).x=a_0(x)$, so that $a_0(x)$ is an idempotent of $M$. Now let $b\in \mathcal I(M)$ be an idempotent such that $x\le b$. Then $n.x\le b$ for each integer $n$, so that $a_0(x)\le b$ which yields (\ref{eq:LSa}).

Since $S(x.n)=S(x)$ for each $n\ge 1$, we have $\bigcup_n S(n.x)\subseteq S(a_0(x))$, which by Proposition \ref{pr:Bel15} means that $S(a_0(x))\setminus \bigcup_n S(n.x)=S(a_0(x))\setminus S(x)$ is a nowhere dense subset of $\mathcal{SM}(M)$. Then $\overline {S(x)}=\overline {S(n.x)} \subseteq S(a_0(x))$.  Because $S(a_0(x))$ is compact and clopen by Theorem \ref{th:Bel11}, $S(a_0(x))\setminus \overline {S(x)}\subseteq S(a_0(x)) \setminus S(x)$, which gives $S(a_0(x))\setminus \overline {S(x)}$ is nowhere dense and open. Lemma \ref{le:LS5} yields, $S(a_0(x))\setminus \overline S(x)=\emptyset$ and $S(a_0(x))= \overline {S(x)}$.

Assume that $b$ is another idempotent of $M$ such that $\overline{S(x)}=S(b)$. First, let $a:=a_0(x)\le b$. Then $b = a\vee \lambda_b(a)$, and $\lambda_b(a)$ is an idempotent of $M$, which entails $s(\lambda_b(a))=0$ for each state-morphism $s$ of $M$. The semisimplicity of $M$ yields $\lambda_b(a)=0$ and $a=b$. In general, we have $S(a)=S(a)\cup S(b)=S(a\vee b)$, i.e. $a=a\vee b = b$.

Let $a = a_0(x)$. Then $a = x\oplus \lambda_a(x)$. By (\ref{eq:LSa}), there is an idempotent $c_0=\bigwedge\{c\in \mathcal I(M) \mid  \lambda_a(x)\le c\}$. Then for the idempotent $\lambda_a(c_0)$ we have $\lambda_a(c_0)=\bigvee \{b\in \mathcal I(M) \mid b \le x\}$. Clearly, $n.\lambda_a(x)\le c_0$, so that $ \lambda_a(c_0)\le x^n$ for each $n\ge 1$, and whence $\lambda_a(c)\le y_0:=\bigwedge_n x^n$. Using \cite[Prop 1.22]{georgescu}, we have $y_0\odot y_0=y_0$ so that $y_0$ is an idempotent of $M$ with $y_0\le x$. Therefore, $y_0\le \lambda_a(c_0)$.

(1) Now let $x\le y$. There is a unique idempotent $b$ of $M$ such that $\overline{S(y)}=S(b)$. Then $S(b)=\overline{S(y)}\supseteq \overline{S(x)}= S(a)$ and $S(b\vee a)=S(b)\cup S(a)= S(b)$, i.e. $a\vee b = b$ and $a \le b$.

(2) By the above parts, the idempotents $a$ and $a_n$ with $\overline{S(x)}=S(a)$ and $\overline{S(x_n)}=S(a_n)$ are determined unambiguously, where $x =\bigvee_n x_n$. Put $a_0=\bigvee_na_n$. Then $a_0\ge a_n\ge x_n$, $a_0\ge x$, so that $a_0\ge a_0(x):=a$.  Now let $b$ be any idempotent of $M$ with $b\ge x$. Then $b\ge x_n$ for each $n\ge 1$, so that $b\ge a_n$ for each $n\ge 1$, and $b\ge a_0$ which by (\ref{eq:LSa}) yields $a_0=a_0(x)=a$.
\end{proof}

The elements $a_0(x)$ and $b_0(x)$ defined in the latter theorem are said to be the {\it least upper idempotent} of $x$ and the {\it greatest lower idempotent} of $x$, respectively, and for them, we have
$$
b_0(x)\le x\le a_0(x).
$$

\begin{prop}\label{pr:LS1}
Let $M$ be a $\sigma$-complete $EMV$-algebra and let $\mathcal B(M)$ be the system of all compact and open subsets of $\mathrm{MaxI}(M)$. Then $\mathcal B(M)=\{M(a)\mid a \in \mathcal I(M)\}$. Moreover, for $a,b \in \mathcal I(M)$, we have $M(a)=M(b)$ if and only if $a=b$, and the closure of the union of countably many elements of $\mathcal B(M)$ belongs to $\mathcal B(M)$.

In particular, for every sequence $\{a_n\}$ of elements of $\mathcal I(M)$,
\begin{equation}\label{eq:LS1}
\overline{\bigcup_n M(a_n)}=M(a),
\end{equation}
where $a = \bigvee_n a_n$ and $a \in \mathcal I(M)$. Similarly, $\overline{\bigcup_n S(x_n)}=S(a)$.
\end{prop}

\begin{proof}
Due to Theorem \ref{th:Bel11}, every $M(a)$ is open and compact for each idempotent $a \in \mathcal I(M)$. Therefore, each $M(a)$ belongs to $\mathcal B(M)$.

If $K$ is a compact and open subset of $\mathrm{MaxI}(M)$, we assert there is an element $x_0 \in M$ such that $K=O(x_0)$. Indeed, we have $K=C(J)=O(I)$ for some ideals $J$ and $I$ of $M$. Since $I=\bigvee\{I_x\mid x \in I\}$, where $I_x$ is the ideal of $M$ generated by an element $x$, then $O(I)=\bigcup \{O(I_x)\mid x \in I\}$, and the compactness of $K$ provides us with finitely many elements $x_1,\ldots,x_n$ of $I$ such that if $x_0 = x_1\vee \cdots \vee x_n\in I$, then $K=O(I)=\bigcup_{i=1}^n O(I_{x_i})= O(I_{x_0})=M(x_0)$. Define $a_0(x_0)$ by (\ref{eq:a(x)}). Then by Theorem \ref{th:LS2}, $K=M(x_0)=\overline{M(x_0)}= M(a_0(x_0))$. From the same theorem, we conclude that for two idempotents $a,b \in \mathcal I(M)$, $M(a)=M(b)$ implies $a=b$.

Now let $\{K_n\}$ be a sequence of elements from $\mathcal B(M)$. For each $K_n$, there is a unique idempotent $a_n \in \mathcal I(M)$ such that $K_n=M(a_n)$. Put $a=\bigvee_n a_n$; then $a\in \mathcal I(M)$. By Proposition \ref{pr:Bel15}, $M(a)\setminus \bigcup_n M(a_n)$ is nowhere dense. Since $M(a)\setminus \overline{\bigcup_n M(a_n)} \subseteq M(a)\setminus \bigcup_n M(a_n)$, the set $M(a)\setminus \overline{\bigcup_n M(a_n)}$ is open and nowhere dense which by Lemma \ref{le:LS5} yields $M(a)\setminus \overline{\bigcup_n M(a_n)}=\emptyset$, i.e. $M(a)= \overline{\bigcup_n M(a_n)}=\overline{\bigcup_n K_n}$.

The second equality $\overline{\bigcup_n S(x_n)}=S(a)$ follows from Theorem \ref{th:Bel11}.
\end{proof}

An important notion of this section is an $EMV$-tribe of fuzzy sets which is a $\sigma$-complete $EMV$-algebra where all operations are defined by points.

\begin{defn}\label{de:tribe}
A system $\mathcal T\subseteq [0,1]^\Omega$ of fuzzy sets of a set $\Omega\ne \emptyset$ is said to be an $EMV$-{\it tribe} if
\vspace{1mm}
\begin{enumerate}[nolistsep]
\item[(i)] $0_\Omega \in \mathcal T$ where $0_\Omega(\omega)=0$ for each $\omega \in \Omega$;
\item[(ii)] if $a \in \mathcal T$ is a characteristic function, then (a) $a-f \in \mathcal T$ for each $f\in \mathcal T$ with $f(\omega)\le a(\omega)$ for each $\omega \in \Omega$, (b) if $\{f_n\}$ is a sequence of functions from $\mathcal T$ with $f_n(\omega)\le a(\omega)$ for each $\omega \in \Omega$ and each $n\ge 1$, then $\bigoplus_n f_n\in \mathcal T$, where $\bigoplus_nf_n(\omega) = \min\{\sum_n f_n(\omega),a(\omega)\}$, $\omega \in \Omega$, and $a$ is a characteristic function from $\mathcal T$;
\item[(iii)] for each $f \in \mathcal T$, there is a characteristic function $a \in \mathcal T$ such that $f(\omega)\le a(\omega)$ for each $\omega \in \Omega$;
\item[(iv)] given $\omega \in \Omega$, there is $f \in \mathcal T$ such that $f(\omega)=1$.
\end{enumerate}
\end{defn}

\begin{prop}\label{pr:LS4}
Every $EMV$-tribe of fuzzy sets is a Dedekind $\sigma$-complete $EMV$-clan where all operations are defined by points. If $\{g_n\}$ is a sequence from $\mathcal T$, then $g=\bigwedge_n g_n$ exists in $\mathcal T$ and $g(\omega)=\inf_n g_n(\omega)$, $\omega \in \Omega$.

If for a sequence $\{f_n\}$ from $\mathcal T$, $f=\bigvee_n f_n$ exists in $\mathcal T$, then $f(\omega)=\sup_n f_n(\omega)$, $\omega \in \Omega$.
An $EMV$-tribe is $\sigma$-complete if and only if, for each sequence $\{f_n\}$ of elements of $\mathcal T$, there is a characteristic function $a\in \mathcal T$ such that $f_n(\omega)\le a(\omega)$, $\omega \in \Omega$.
\end{prop}

\begin{proof}
By \cite[Prop 4.10]{DvZa}, we see that $\mathcal T$ is an $EMV$-clan of fuzzy sets of $\Omega$ which is closed under $\vee$ and $\wedge$, defined by points.  We have to show that the operation $\bigoplus$ is correctly defined. Let $\{f_n\}$ be any sequence for which there are two characteristic functions $a,b \in \mathcal T$ such that $f_n(\omega)\le a(\omega),b(\omega)$, $\omega \in \Omega$ and $n\ge 1$. There is another characteristic function $c\in \mathcal T$ with $a(\omega),b(\omega)\le c(\omega)$, $\omega \in \Omega$. We denote $(\bigoplus^a_n f_n)(\omega):=\min\{\sum_n f_n(\omega),a(\omega)\}$ for each $\omega\in \Omega$. In the same way we define $\bigoplus^b_n f_n$ and $\bigoplus^c_n f_n $. Then

\[(\bigoplus^a_n f_n)(\omega)=\begin{cases}
\sum_n f_n(\omega) & \text{ if } \sum_n f_n(\omega)\le a(\omega)\\
a(\omega) & \text{ if }  \sum_n f_n(\omega) > a(\omega),
\end{cases}\quad \omega \in \Omega,
\]
and
\[(\bigoplus^c_n f_n)(\omega)=\begin{cases}
\sum_n f_n(\omega) & \text{ if } \sum_n f_n(\omega)\le c(\omega)\\
c(\omega) & \text{ if }  \sum_n f_n(\omega) > c(\omega),
\end{cases}\quad \omega \in \Omega.
\]

If $a(\omega)=0$, then $f_n(\omega)=0$ for each $n$ and $(\bigoplus^a_n f_n)(\omega)= 0= (\bigoplus^c_n f_n)(\omega)$. If $a(\omega)=1$, then $c(\omega)=1$ and $(\bigoplus^a_n f_n)(\omega)=  (\bigoplus^c_n f_n)(\omega)$. In the same way we have $(\bigoplus^b_n f_n)=(\bigoplus^c_n f_n)$, so that $(\bigoplus^a_n f_n)=(\bigoplus^b_n f_n)$, and $\bigoplus_n f_n$ is well defined.

Choose an arbitrary sequence $\{f_n\}$ from $\mathcal T$ which is dominated by some characteristic function $a\in \mathcal T$. Without loss of generality we can assume that $f_n(\omega)\le f_{n+1}(\omega)$, $\omega \in \Omega$, $n\ge 1$. We set $h_1=f_1$ and $h_n=f_n-f_{n+1}$ for $n\ge 1$. Then each $h_n$ belongs to $\mathcal T$ and it is dominated by $a$. Therefore, $\bigoplus_n h_n \in \mathcal T$ and $(\bigoplus_n h_n)(\omega) = \sum_n h_n(\omega)= \sup_nf_n(\omega)$, which proves that $\mathcal T$ is Dedekind $\sigma$-complete. Consequently, $\mathcal T$ is $\sigma$-complete iff for each sequence $\{f_n\}$ we can find a characteristic function $a \in \mathcal T$ which dominates each $f_n$.

Now let $\{g_n\}$ be any sequence from $\mathcal T$. Since $\mathcal T$ is a lattice where $(f\wedge g)(\omega)=\min\{f(\omega),g(\omega)\}$, $\omega \in \Omega$, without loss of generality, we can assume that $g_{n+1}\le g_n$ for each $n\ge 1$. Then there is a characteristic function $a\in \mathcal T$ such that $g_n(\omega) \le a(\omega)$, $\omega \in \Omega$, $n \ge 1$, and $a-g_n\in \mathcal T$, $a-g_n \le a-g_{n+1}$. Whence, $(\bigvee_n (a-g_n))(\omega)= \sup_n(a-g_n)(\omega)$ for each $\omega \in \Omega$. Consequently $(\bigwedge_n g_n)(\omega)= a(\omega)-(\bigvee_n (a-g_n))(\omega)= a(\omega)-\sup_n(a(\omega)-g_n(\omega))= \inf_n g_n(\omega)$, $\omega \in \Omega$.
\end{proof}

We note that a {\it tribe} is a system $\mathcal T\subseteq [0,1]^\Omega$ of fuzzy sets on $\Omega\ne \emptyset$ such that (i) $1_\Omega \in \mathcal T$, (ii) if $f \in \mathcal T$, then $1-f\in \mathcal T$, and (iii) for any sequence $\{f_n\}$ of elements of $\mathcal T$, the function $\bigoplus_n f_n$ belongs to $\mathcal T$, where $(\bigoplus_nf_n)(\omega)=\min\{\sum_n f_n(\omega),1\}$, $\omega \in \Omega$. Then the notion of an $EMV$-tribe is a generalization of the notion of a tribe because an $EMV$-tribe $\mathcal T$ is a tribe iff $1_\Omega \in \mathcal T$. We note that in \cite{Mun4,Dvu5}, there was proved that every $\sigma$-complete $MV$-algebra is a $\sigma$-homomorphic image of some tribe of fuzzy sets.

We say that an $EMV$-homomorphism $h: M_1\to M_2$ is a $\sigma$-{\it homomorphism}, where $M_1$ and $M_2$ are $EMV$-algebras, if for any sequence $\{x_n\}$ of elements from $M_1$ for which $x=\bigvee_n x_n$ is defined in $M_1$, then $\bigvee_n h(x_n)$ exists in $M_2$ and $h(x)=\bigvee_n h(x_n)$.

Let $f$ be a real-valued function on $\Omega \ne \emptyset$. We define
$$
N(f) :=\{\omega \in \Omega \mid |f(\omega)| > 0\}, \quad N^+(f)=\{\omega \in \Omega \mid f(\omega)>0\}, \quad N^-(f)=\{\omega \in \Omega \mid f(\omega)<0\}.
$$
Then $N(f)=N^+(f)\cup N^-(f)$.

Suppose that $\mathcal T$ is a system of fuzzy sets on $\Omega$, containing $0_\Omega$, such that, for each $f\in \mathcal T$, there is a characteristic function $a\in \mathcal T$ with $f(\omega)\le a(\omega)$, $\omega \in \Omega$. If $f,g \le a$ for some characteristic function from $\mathcal T$, we can define $(f\oplus g)(\omega) =\min\{f(\omega)+g(\omega),a(\omega)\}$, $(f\odot g)(\omega) = \max\{f(\omega)+g(\omega)-a(\omega),0\}$, and $(f\ast g)(\omega) = \max\{f(\omega)-g(\omega),0\}$ for each $\omega \in \Omega$, and these operations do not depend on $a$.

Then for all $f,g \in \mathcal T$ we have

\begin{itemize}[nolistsep]
\item[(i)]    $N(f\oplus g)= N(f) \cup N(g)$;
\item[(ii)]   $N(f \ast g) =\{\omega \in \Omega\mid f(\omega) > g(\omega)\}$;
\item[(iii)]  $(f \ast g) \oplus (g \ast f) = (f \ast g) + (g \ast f)$;
\item[(iv)]   $N((f \ast g) \oplus (g \ast f)) = N(f-g)$;
\item[(v)]    $N(f) \subseteq N(g)$ if $f \le g$;
\item[(vi)] $N(f\ast g)=\{\omega\in \Omega \mid f(\omega)>g(\omega)\}$;
\item[(vii)] $N(f\odot g)= \{\omega\in \Omega \mid f(\omega)+g(\omega)>1\}$.
\end{itemize}
\vspace{1mm}

Now we formulate the Loomis--Sikorski theorem for $\sigma$-complete $EMV$-algebras.

\begin{thm}\label{th:LS6}{\bf (Loomis--Sikorski Theorem)}
Let $M$ be a $\sigma$-complete $EMV$-algebra. Then there are an $EMV$-tribe $\mathcal T$ of fuzzy sets on some $\Omega \ne \emptyset$ and a surjective $\sigma$-homomorphism $h$ of $EMV$-algebras  from $\mathcal T$ onto $M$.
\end{thm}

\begin{proof}
If $M=\{0\}$, the statement is trivial. So let $M\ne \{0\}$.

By Proposition \ref{pr:Bel3}, $M$ is a semisimple $EMV$-algebra, and by the proof of \cite[Thm 4.11]{DvZa}, $M$ is isomorphic to $\widehat M=\{\hat x\mid x \in M\}$, where $\hat x: \mathcal{SM}(M)\to [0,1]$ is defined by $\hat x(s)=s(x)$, $s \in \mathcal{SM}(M)$.

Let $\mathcal T$ be the system of fuzzy sets $f$ on $\Omega=\mathcal{SM}(M)$ such that (i) for some $x \in M$, $N(f - \hat x)$ is a meager set (i.e. it is a countable union of nowhere dense subsets) in the weak topology of state-morphisms, and we write $f \sim x$, and (ii) there is $a \in \mathcal I(M)$ such that $f\le \hat a$. It is clear that $\mathcal T$ contains $\widehat M$.

If $x_1$ and $x_2$ are two elements of $M$ such that $N(f-\hat x_i)$ is a meager set for $i=1,2$, then
$$ N(\hat x_1 - \hat x_2) \subseteq N(\hat x_1 - f) \cup N(f - \hat x_2)$$
is a meager set. By Lemma \ref{le:LS5}, we conclude that $N(\hat x_1 - f) \cup N(f - \hat x_2)=\emptyset$ from which we get $\hat x_1 = \hat x_2$, that is $x_1=x_2$. Therefore, if $f \sim x_1$ and $f\sim x_2$, then $x_1=x_2$.

\vspace{2mm}
\noindent {\it Claim 1.
The set $\mathcal T$ is an $EMV$-clan.}
\vspace{1mm}

Let $f,g,h\in \mathcal T$ and let $N(g-h)$ be a meager set. We assert $N_0:=N((f\oplus g)\ast (f\oplus h))$ is a meager set. Set $N_1=\{s \mid \min\{f(s)+g(s),1\}> \min\{f(s)+h(s),1\}\}$ and check
\begin{eqnarray*}
N_1&=&  \big(N_1\cap \{s\mid g(s)=h(s)\}\big) \cup \big(N_1 \cap \{s\mid g(s)>h(s)\}\big) \cup \big(N_1 \cap \{s \mid g(s)<h(s)\}\big)\\
&=& \big(N_1 \cap \{s\mid g(s)>h(s)\}\big) \cup \big(N_1 \cap \{s \mid g(s)<h(s)\}\big)\subseteq N_1\cap N(g-h),
\end{eqnarray*}
which shows that $N_0$ is a meager set. Similarly, $N((f\oplus h)\ast (f\oplus g))$ is a meager set.

In a similar way, if $N_3:= N((f\vee g)\ast (f\vee h))=\{s \mid f(s)\vee g(s) > f(s)\vee h(s)\}$, then
\begin{eqnarray*}
N_3&=& \big(N_3\cap \{s\mid g(s)=h(s)\}\big) \cup \big(N_3 \cap \{s\mid g(s)>h(s)\}\big) \cup \big(N_3 \cap \{s \mid g(s)<h(s)\}\big)\\
&=& \big(N_3 \cap \{s\mid g(s)>h(s)\}\big) \cup \big(N_3 \cap \{s \mid g(s)<h(s)\}\big) \subseteq N_3\cap N(g-h)\subseteq N(g-h),
\end{eqnarray*}
which establishes $N_3$ is a meager set. In the same way, $N((f\vee h)-(f\vee g))$ is meager, consequently, $N((f\vee g)- (f\vee h))$ is meager.

Therefore, if $f,g \in \mathcal T$ and $f\sim x$ and $g\sim y$ for unique $x,y \in M$, there is an idempotent $a\in \mathcal I(M)$ such that $x,y\le a$ and $f,g \le \hat a$. This implies $N((f\oplus g)\ast (\hat x \oplus \hat y)) \subseteq N((f\oplus g)\ast (f\oplus \hat y))\cup N((f\oplus \hat y) \ast (\hat x\oplus \hat y))$ is a meager set. Similarly  $N((\hat x \oplus \hat y)\ast (f\oplus g))$ is also meager. Therefore, $f\oplus g \sim x\oplus y$ which proves $\mathcal T$ is an $EMV$-clan and $\mathcal T$ is closed also under $\vee$ and $\wedge$ with pointwise ordering. In the same way, we have also $f\vee g \sim x\vee y$.

We note that if  $f\in \mathcal T$ is a characteristic function such that $f \sim x \in M$, $f \le \hat a$ for some $a \in \mathcal I(M)$, then $f=f\oplus f \sim x\oplus x = x$, so that $x$ is an idempotent of $M$.

Let $f \in \mathcal T$, $f\sim x$, $f\le b$ for some characteristic function $b\in \mathcal T$. Then there is a unique idempotent $a \in \mathcal I(M)$ such that $b\sim a$, in addition, $x\le a$. Then we have $\widehat{\lambda_a(x)}=\hat a - \hat x$, and
$$
N((b-f)-(\widehat{\lambda_a(x)}))= N((b-f)-(\hat a- \hat x))=N((b-\hat a)-(f-\hat x))\subseteq N(b-\hat a) \cup N(f-\hat x),
$$
which is a meager set. Hence,
\begin{equation}\label{eq:dif}
\lambda_b(f)=b-f \sim \lambda_a(x).
\end{equation}

We note that if $f,g \in \mathcal T$, and if $a$ is an idempotent of $M$ such that $f,g \le \hat a$, then $1-f, f\vee g, f\oplus g$ are dominated by $\hat a$. Consequently, $\mathcal T$ is an $EMV$-clan.

\vspace{2mm}
\noindent {\it Claim 2.
The set $\mathcal T$ is closed under pointwise limits of non-decreasing sequences from $\mathcal T$.}
\vspace{1mm}

Let $\{f_n\}_n$ be a sequence of non-decreasing functions from
$\mathcal T$. Choose $x_n \in M$ such that $f_n \sim x_n$ for each $n\ge 1$.
Since $f_n = f_1 \vee \cdots \vee f_n \sim x_1\vee \cdots \vee x_n$ for each $n\ge 1$, we have $x_n \le x_{n+1}$. Denote $f = \lim_n f_n,$ $x = \bigvee_{n=1}^\infty x_n$, and
$ b_0 = \lim_n \hat x_n.$ Then $x \in M$.  It is easy to see that there is an idempotent $a$ such that $x,x_1,\ldots \le a$ and $f_1,f_a\le \hat a$.

We have
$$
N(f - \hat x) \subseteq N(f - b_0) \cup N(\hat x -  b_0)
$$
and $N(f - b_0) = \{s \mid f(s) <  b_0(s)\} \cup
\{s \mid\  b_0(s) < f(s)\}$.

If $s \in  \{s \mid\ f(s) < b_0(s)\},$ then
there is an integer $n \ge 1$ such that $f(s) < \hat
x_n(s) \le b_0(s)$. Hence, $f_n(s) \le
f(s) < \hat x_n(s) \le b_0(s)$ so that $s
\in \{s \mid\ f_n(s) < \hat x_n(s)\}$.

Similarly we can prove that if $s\in \{s\mid b_0(s)
< f(s)\}$, then there is an integer $n \ge 1$ such that
$s\in \{s \mid\ \hat b_n(s) < f_n(s)\}$.

The last two cases imply
$$ N(f - b_0) \subseteq \bigcup_{n=1}^\infty N(\hat x_n - f_n)
$$
which is a meager set.

Now it is necessary to show that $N(\hat x-b_0)$ is a meager set. We have
\begin{eqnarray*}
N(\hat x-b_0)&=& \big(N(\hat x-b_0)\cap \{s\mid s(x)>0\}\big)\cup \big(N(\hat x-b_0)\cap \{s\mid s(x)=0\}\big)=N(\hat x-b_0)\cap S(x)\\
&=&  \big(N(\hat x-b_0) \cap (S(x)\setminus \bigcup_n S(x_n))\big) \cup \big(N(\hat x-b_0) \cap (S(x)\cap \bigcup_n S(x_n))\big).
\end{eqnarray*}
By Proposition \ref{pr:Bel15}, we have $N(\hat x-b_0) \cap (S(x)\setminus \bigcup_n S(x_n))$ is a meager set. Therefore, it is necessary to prove that $N:=N(\hat x-b_0) \cap (S(x)\cap \bigcup_n S(x_n))=N(\hat x-b_0) \cap\bigcup_n S(x_n)$ is a meager set.

To prove it, take an arbitrary open non-empty set $O$ in $\mathcal{SM}(M)$. Then there is an ideal $I$ of $M$ such that $O=\{s\in \mathcal{SM}(M)\mid I \varsubsetneq \Ker(s)\}$. The ideal $I$ contains a non-zero element $z\in I$. There is an idempotent $a\in \mathcal I(M)$ such that $x,z\le a$. We note that in such a case, $a_0(x)\le a$, where $a_0(x)$ is the least upper idempotent of $x$ defined in Theorem \ref{th:LS2}. The restriction of any state-morphism $s \in S(a)$ onto the $MV$-algebra $M_{a}=[0,a]$ is a state-morphism on $M_{a}$; we denote the set of those restrictions by $S_0(a)$. Then $S_0(a)\subseteq \mathcal{SM}(M_{a})$. It is clear that $M_{a}$ is a $\sigma$-complete $MV$-algebra, whence $x,x_1,\ldots \in M_{a}$ and $x$ is the least upper bound of $\{x_n\}$ taken in the $MV$-algebra $M_{a}$. By the proof of \cite[Thm 4.1]{Dvu5},  $S_0:=\{s\in \mathcal{SM}(M_{a})\mid s(x) > \lim_n s(x_n)\}$ is a meager set in the weak topology of $\mathcal{SM}(M_{a})$. Then $\{s_{|M_{a}}\mid s \in S(x) \cap N(\hat x -b_0)\}\subseteq S_0$ is also a meager set of $\mathcal{SM}(M_{a})$.

The element $z$ belongs to $[0,a]$, and let $I_a=I\cap[0,a]$. Clearly $I_a$ is an ideal of $M_a$ containing $z$,  and let $O_a(I_a)=\{s\in \mathcal{SM}(M_a)\mid I_a\varsubsetneq \Ker(s)\}$. Then $O_a(I_a)$ is a non-zero open set of $\mathcal{SM}(M_a)$. Therefore, there is an element $0<y\in M_a$ such that $S_a(y)=\{s \in \mathcal{SM}(M_a)\mid s(y)>0\}\subseteq O_a(I_a)$ and it has the empty intersection with $S_0$. Define $S(y)=\{\mathcal{SM}(M) \mid s(y)>0\}$.  Since $y\le a$, we have $S(y)\subseteq M(a)$. For each state-morphism $s$ on $M$, let $s_a$ be the restriction of $s$ onto $M_a$. Take $s \in S(y)$, then $s_a(y)>0$, $s_a$ is a state-morphism on $M_a$,  $s_a\in S_a(y)$, and $s_a \in O_a(I_a)$. That is, there is a non-zero element $t\in I_a$ such that $s_a(t)=0$, i.e. $s(t)=0$ for some $t \in I$ which gets $s\in O$. We have proved that $S(y)\subseteq O$. We assert $S(y)\cap S(x)\cap N(\hat x- b_0)=\emptyset$. If not, there is a state-morphism $s$ belonging to the intersection. Then $s(a)=1$ since $s \in S(y)$, so that $s_a$ is a state-morphism on $M_a$, $s_a(y)=s(y)>0$, and $\hat x(s)-b_0(s)=s_a(x)-\lim_ns_a(x_n)>0$ which is an absurd, and the intersection is empty. Therefore, the set $S(x)\cap N(\hat x- b_0)$ is a meager set.

Hence, given a non-decreasing sequence $\{f_n\}$, for the function $f$ defined by $f(s)=\sup_n f_n(s)$, $s \in \mathcal{SM}(M)$, we have $f \sim x$, where $x=\bigvee_n x_n$, and clearly $f \in \mathcal T$.

\vspace{2mm}
\noindent {\it Claim 3.
The set $\mathcal T$ is an $EMV$-tribe.}
\vspace{1mm}

Now let $\{f_n\}$ be an arbitrary sequence of functions from $\mathcal T$ such that $f_n \sim x_n \in M$. By the previous step, there is an idempotent $a\in M$ such that $x_1,x_2,\ldots \le a$ and $f_1,f_2,\ldots \le \hat a$.
Then for each $n\ge 1$, $g_n=f_1\oplus\cdots \oplus f_n=\min\{f_1+\cdots + f_n,\hat a\}\sim x_1\oplus\cdots \oplus x_n$ and it does not depends on $a$.
Then $\bigoplus_n f_n$ is a pointwise limit of the non-decreasing sequence $\{g_n\}$, that is, $\bigoplus_n f_n= \lim_n g_n$, which by Claim 2 means, $\bigoplus_n f_n \sim \bigvee_n(x_1\oplus \cdots \oplus x_n)$. In addition, $\bigoplus_n f_n \le \hat a$, so that, we have shown that $\bigoplus_n f_n\in \mathcal T$ and $\mathcal T$ is an $EMV$-tribe of fuzzy sets on $\mathcal{SM}(M)$. Since by the construction of $\mathcal T$, for each $f \in \mathcal T$, there is an idempotent $a\in \mathcal I(M)$ such that $f\le \hat a$,
Proposition \ref{pr:LS4} says that $\mathcal T$ is an $EMV$-tribe.

\vspace{2mm}
\noindent {\it Claim 4. $M$ is a $\sigma$-homomorphic image of the $EMV$-tribe
 $\mathcal T$.}
\vspace{1mm}

Define a mapping $h:\mathcal T\to M$ by $h(f)=x$ iff $f \in \mathcal T$ and $f\sim x \in M$. By the first part of the present proof, $h$ is a well-defined mapping that is surjective. It preserves $\oplus, \vee,\wedge$, and $h(0_\Omega)=0$. In addition, if $f = \bigvee_n f_n=\sup_n f_n$, then by Step 2, $f_n\sim x_n$ and $f\sim x=\bigvee_n x_n$, that is $h(f)=\bigvee_n h(f_n)$.

Now let $f\le b$, where $f\in \mathcal T$ and $b$ is a characteristic function from $\mathcal T$. There are unique elements $x\in M$ and $a\in \mathcal I(M)$ such that $f\sim x$ and $b\sim a$. Clearly, $x\le a$. Then $b = f \oplus \lambda_b(f)$, and by (\ref{eq:dif}), we have $b-f \sim \lambda_a(x)$, i.e. $h(b-f)=h(\lambda_b(f))=\lambda_a(x)$,
so that $a=h(b)=h(f)\oplus h(\lambda_b(f))= h(f) \oplus \lambda_{h(b)}(h(f))= x\oplus \lambda_a(x)$. By definition of $\lambda_{h(b)}$ in $M$, we have $\lambda_a(x)=\lambda_{h(b)}(h(f))\le h(\lambda_b(f))=\lambda_a(x)$, that is $h(\lambda_b(f))=\lambda_{h(b)}(h(f))$, which proves that $h$ is a homomorphism of $EMV$-algebras. Consequently, $h$ is a surjective $\sigma$-homomorphism as we needed.

Theorem is proved.
\end{proof}

We recall that if $\Omega$ is a non-void set, then a {\it ring} is a system $\mathcal R$ of subsets of $\Omega$ such that (i) $\emptyset \in \mathcal R$, (ii) if $A,B \in \mathcal R$, then $A\cup B, A\setminus B\in \mathcal R$. A ring $\mathcal R$ is a $\sigma$-{\it ring} if given a sequence $\{A_n\}$ of subsets from $\mathcal R$, $\bigcup_n A_n \in \mathcal R$. Clearly, every ring is an $EMV$-algebra and a generalized Boolean algebra of subsets.

We remind that due to the Stone theorem, see e.g. \cite[Thm 6.6]{LuZa}, every generalized Boolean algebra is isomorphic to some ring of subsets.

A corollary of the Loomis--Sikorski Theorem \ref{th:LS6} is the following result.

\begin{cor}\label{co:LS7}
Let $M$ be a $\sigma$-complete $EMV$-algebra. Then there are a $\sigma$-ring $\mathcal R$ of subsets of some set $\Omega\ne \emptyset$ and a surjective $\sigma$-homomorphism from $\mathcal R$ onto $\mathcal I(M)$.
\end{cor}

\begin{proof}
Since $M$ is $\sigma$-complete, by Proposition \ref{pr:Bel3}, $\mathcal I(M)$ is a $\sigma$-complete subalgebra of $M$, in other words, $\mathcal I(M)$ is a $\sigma$-complete generalized Boolean algebra.

Use the system $\mathcal T$ defined in the proof of Theorem \ref{th:LS6}, that is $f \in \mathcal T$ iff there is an element $x\in M$ with $f\sim x$ and there is an idempotent $a\in M$ such that $f\le \hat a$; $\mathcal T$ is a $\sigma$-complete $EMV$-tribe of fuzzy functions on $\Omega =\mathcal{SM}(M)$. Then the mapping $h: \mathcal T \to M$ defined by $h(f)=x$ ($f\in \mathcal T$) if $f\sim x\in M$, is a surjective $\sigma$-homomorphism.

Denote by $\mathcal R_0$ the class of all characteristic functions from $ \mathcal T$. As it was proved in Theorem \ref{th:LS6}, for each $f \in \mathcal R_0$, there is a unique $x\in \mathcal I(M)$ such that $f \sim x$. If (i) $\chi_A,\chi_B \in \mathcal R_0$, then $\chi_A\vee \chi_B=\chi_A \oplus \chi_B = \chi_{A\cup B}\in \mathcal R_0$, (ii) if $\chi_A,\chi_B \in \mathcal R_0$ and $\chi_A \le \chi_ B$, then $\chi_B - \chi_A \in \mathcal R_0$, (iii) if $\chi_A,\chi_B \in \mathcal R_0$, then $\chi_A \wedge \chi_B = \chi_{A\cap B}\in \mathcal R_0$, and (iv) if $\{\chi_{A_n}\}$ is a sequence of characteristic functions from $\mathcal R_0$, then $\bigoplus_n\chi_{A_n} =\chi_A\in \mathcal R_0$, where $A =\bigcup_n A_n$.

We note here, that in Claim 2 of the proof of the Loomis--Sikorski Theorem, it was necessary to prove that $N(\hat x - b_0)$ is a meager set. We show that if the non-decreasing sequence $\{x_n\}$ of elements of $M$ with $x = \bigvee_n x_n $ and $b_0=\lim_n \hat x_n$ consists only of idempotent elements, the proof of the fact $N(\hat x - b_0)$ is meager is very easy. Indeed, if $s \in N$, there is an integer $n_0$ such that $s\in S(x_{n_0})$. Then we have  $1\ge s(x)\ge s(x_{n_0})=1$ that yields $\hat x(s)=1=b_0(s)$ and the set $N$ is empty.

Now if $h_0:\mathcal R_0 \to \mathcal I(M)$ is the restriction of the $\sigma$-homomorphism $h:\mathcal T\to M$ onto $\mathcal R_0$ we see that $h_0$ is a $\sigma$-homomorphism from $\mathcal R_0 $ onto $\mathcal I(M)$. Now let $\mathcal R=\{A \subseteq \Omega \mid \chi_A \in \mathcal R_0\}$. Then $\mathcal R_0$ is a $\sigma$-complete ring of subsets of $\Omega = \mathcal{SM}(M)$. Define a mapping $\iota: \mathcal R\to \mathcal R_0$ by $\iota(A) = \chi_A$, $A \in \mathcal R$. It is clear that $\iota$ is a $\sigma$-complete isomorphism. If we set $\phi=h_0\circ \iota: \mathcal R \to \mathcal I(M)$, then $\phi$ is a surjective $\sigma$-homomorphism from $\mathcal R$ onto the set of idempotents $\mathcal I(M)$, and the corollary is proved.
\end{proof}

We note that the last result can be found in \cite[p. 216]{Kel} using the language of $\sigma$-complete Boolean rings. Therefore, Theorem \ref{th:LS6} is a generalization of the Loomis-Sikorski Theorem for Boolean $\sigma$-algebras, see \cite{Loo,Sik}, $\sigma$-complete Boolean rings, \cite{Kel},  and $\sigma$-complete $MV$-algebras, see \cite{Mun4, Dvu5,BaWe}.

We say that an ideal $I$ of an $EMV$-algebra $M$ is $\sigma$-{\it complete} if, for each sequence $\{x_n\}$ of elements of $I$, the existence of $\bigvee_n x_n$ in $M$ implies $\bigvee_n x_n \in I$.

\begin{thm}\label{th:LS8}
Every $\sigma$-complete $EMV$-algebra $M$ without top element can be embedded into a $\sigma$-complete $MV$-algebra $N$ as its maximal ideal which is also $\sigma$-complete. Moreover, this $N$ can be represented as
$$
N=\{x\in N\mid \text{ either } x \in M \text{ or } x = \lambda_1(y) \text{ for some } y \in M\}.
$$
\end{thm}

\begin{proof}
If a $\sigma$-complete $EMV$-algebra $M$ possesses a top element, then it is an $MV$-algebra, so $M$ is a $\sigma$-complete $MV$-algebra. Thus, let $M$ have no top element. According to Theorem \ref{th:embed}, there is an $MV$-algebra $N$ such that $M$ can be embedded into $N$ as its maximal ideal. Without loss of generality let us assume that $M$ is an $EMV$-subalgebra of $N$. Let $1$ be the top element of $N$. By the proof of Theorem \ref{th:embed}, every element $x\in N$ is either from $M$, or $\lambda_1(x)\in M$. Due to Mundici's result, see \cite{Mun}, there is a unital Abelian $\ell$-group $(G,u)$ such that $N=\Gamma(G,u)$ so that $1=u$. Thus let $\{x_n\}$ be an arbitrary sequence of elements of $N$.

There are three cases. (1) Every $x_n \in M$.  Then there is an element $x=\bigvee_n x_n\in M$, where the supremum $x$ is taken in the $\sigma$-complete $EMV$-algebra $M$. Thus let $x_n \le y$ for each $n$, where $y\in N$. It is enough to assume that $y=\lambda_1(y_0)$ for some $y_0 \in M$. Using the Mundici representation of $MV$-algebras by unital $\ell$-groups, we obtain $x_n\le \lambda_1(y_0)= u-y_0$, so that $y_0+x_n\le u$, where $+$ and $-$ denote the group addition and the group subtraction, respectively, taken in the group $(G,u)$. Hence, $y_0+x_n = y_0\oplus x_n\in M$, so that there is $\bigvee_n(y_0\oplus x_n)$ in $M$, which means $y_0\oplus \bigvee_n x_n = \bigvee_n (y_0\oplus x_n)\le u$ as well as  $y_0+ \bigvee_n x_n = \bigvee_n (y_0+ x_n)= \bigvee_n (y_0\oplus x_n)\le u$. Then $\bigvee_n x_n \le u -y_0=y$ which proves $\bigvee_nx_n$ is also a supremum of $\{x_n\}$ taken in the whole $MV$-algebra $N$.

We note that for each sequence $\{z_n\}$ of elements of $M$, there is an idempotent $a\in M$ such that $z_n \le a$, so that $z=\bigwedge_n z_n$ exists in $M$ and similarly as for $\bigvee$, we can show that $z$ is also the infimum taken in the whole $N$.

Case (2), every $x_n = \lambda_1({x^0_n})=u-x^0_n$, where $x^0_n \in M$ for each $n\ge 1$. Clearly, $\bigwedge_nx_n$ exists in $M$ as well as in $N$ and they are the same. Hence, in the unital $\ell$-group as well as in the $MV$-algebra $N$, we have $u-\bigwedge_n x^0_n = \bigvee_n(u-x^0_n)= \bigvee_n x_n\in N$ which says $\bigvee_n x_n$ exists in $N$.

Case (3), the sequence $\{x_n\}$ can be divided into two sequences $\{y_i\}$ and $\{z_m\}$, where $y_i \in M$, $z_m = \lambda_1(z^0_m)$ with $z^0_m \in M$ for each $n$ and $m$. By cases (1) and (2), $y=\bigvee_i y_i$ and $z=\bigvee_m z_m$ are defined in $N$, so that $y\vee z$ exists in $N$ and clearly, $y\vee z=\bigvee_n x_n$.

Combining (1)--(3), we see that $N$ is a $\sigma$-complete $MV$-algebra.

From Theorem \ref{th:embed}, we conclude $M$ is a maximal ideal of $N$, and Case (1) says that $M$ is a $\sigma$-ideal of $N$.
\end{proof}

We mention that if $M$ is a $\sigma$-complete $MV$-algebra, then $\mathcal{SM}(M)$ is a basically disconnected space, see \cite[Prop 4.3]{Dvu5}. A similar result holds also for $\sigma$-complete $EMV$-algebras as it follows from the following statement.

\begin{thm}\label{th:bdisc}
Let $M$ be a $\sigma$-complete $EMV$-algebra. If $\{C_n\}$ is a sequence of compact subsets of $\mathcal{SM}(M)$ such that $A=\bigcup_n C_n$ is open, then the closure of $A$ in the weak topology of state-morphisms on $M$ is open.
\end{thm}

\begin{proof}
If $M$ has a top element, the statement follows from \cite[Prop 4.3]{Dvu5}.
Thus let $M$ have no top element and let $A=\bigcup_n C_n$ be open, where each $C_n$ is compact. Let $N$ be an $MV$-algebra representing the $EMV$-algebra given by Theorem \ref{th:embed}.  According to Theorem \ref{th:comp}, the state-morphism space $\mathcal{SM}(N)$ is the one-point compactification of $\mathcal{SM}(M)$, and the mapping $\phi: \mathcal{SM}(M)\to \mathcal{SM}(N)$ defined by $\phi(s)=\tilde s$, $s\in \mathcal{SM}(M)$, given by (\ref{th:embed}), is a continuous embedding of $\mathcal{SM}(M)$ into $\mathcal{SM}(N)$. Then $\mathcal{SM}(N)= \phi(\mathcal{SM}(M))\cup \{s_\infty\}$. We have $\phi(A)=\bigcup_n \phi(C_n)$. Since $s_\infty \notin \phi(A)$, we see that $\phi(A)$ is open and every $\phi(C_n)$ is closed in the weak topology of state-morphisms on $N$. Since $N$ is by Theorem \ref{th:LS8} a $\sigma$-complete $MV$-algebra, the state-morphism space $\mathcal{SM}(N)$ is basically disconnected. That is, $\overline{\phi(A)}^N$ is an open set, where $\overline{K}^N$ and $\overline{K}^M$ denote the closure of $K$ in the weak topology on $\mathcal{SM}(N)$ and $M$, respectively.
If $s_\infty \notin \overline{\phi(A)}^N$, then $\phi^{-1}(\overline{\phi(A)}^N\cap \phi(X))= \overline{A}^M$, where $X=\mathcal{SM}(M)$, which means that $\overline{A}^M$ is open. If $s_\infty \in \overline{\phi(A)}^N$, then $\overline{\phi(A)}^N=\phi(\overline{A}^M)\cup \{s_\infty\}$, so that $X\setminus \phi^{-1}(\overline{\phi(A)}^N)= X\setminus \overline{A}^M$ is compact, and $\overline{A}^M$ is open.
\end{proof}

Now we present another proof of the Loomis--Sikorski theorem for $\sigma$-complete $EMV$-algebras which is based on Theorem \ref{th:LS8} and on the Loomis--Sikorski representation of $\sigma$-complete $MV$-algebras, see e.g. \cite{Dvu5,Mun4}. We note that the proof from Theorem \ref{th:LS6} gives an interesting and more instructive look into important topological methods which follow from the hull-kernel topology of maximal ideals and the weak topology of state-morphisms than a simple application of the Loomis--Sikorski theorem for $\sigma$-complete $MV$-algebras.

\begin{thm}\label{th:LS9}{\bf (Loomis--Sikorski Theorem 1)}
Let $M$ be a $\sigma$-complete $EMV$-algebra. Then there are an $EMV$-tribe $\mathcal T$ of fuzzy sets on some $\Omega \ne \emptyset$ and a surjective $\sigma$-homomorphism $h$ of $EMV$-algebras from $\mathcal T$ onto $M$.
\end{thm}

\begin{proof}
Let $M$ be a proper $\sigma$-complete $EMV$-algebra. According Theorem \ref{th:LS8}, $M$ can be embedded into a $\sigma$-complete $MV$-algebra $N$ as its maximal ideal which is also $\sigma$-complete. Without loss of generality, we can assume that $M$ is an $EMV$-subalgebra of $N$, and every element $x$ of $N$ is either from $M$ or $\lambda_1(x)$ is from $M$. Using the famous Mundici representation of $MV$-algebras by unital $\ell$-groups, there is a unital Abelian $\ell$-group $(G,u)$ such that $N=\Gamma(G,u)$. Hence, if $x\le a\in \mathcal I(M)$, then $\lambda_a(x)=a-x$, where $-$ is the subtraction taken from the $\ell$-group $G$.

By \cite[Thm 5.1]{Dvu5}, there are a tribe $\mathcal T_0$ of fuzzy sets of some set $\Omega\ne \emptyset$ and a $\sigma$-homomorphism of $MV$-algebras $h_0$ from $\mathcal T_0$ onto $N$. We note that if $\{f_n\}$ is a sequence of functions from $\mathcal T_0$ such that there is a characteristic function $a\in \mathcal T_0$ with $f_n(\omega)\le a(\omega)$ for each $\omega \in \Omega$ and each integer $n$, then
$$
\min\{\sum_n f_n(\omega),a(\omega)\}=\min\{\sum_n f_n(\omega),1\},\quad \omega \in \Omega.
$$
This statement follows the same proof of an analogical equality from the proof of Proposition \ref{pr:LS4}. Therefore, $h_0(f\oplus g)=h_0(f)\oplus h_0(g)$. Let $f\in \mathcal T_0$ and assume that $a$ is a characteristic function from $\mathcal T_0$ such that $f \le a$. Then $\lambda_a(f)=a-f \in \mathcal T_0$ and $a = f+ (a-f)=f \oplus (a-f)$ which means $h_0(a)=h_0(f)\oplus h_0(a-f)= h_0(f)+ (h_0(a)-h_0(f))= h_0(f)+\lambda_{h_0(a)}(h_0(f))= h_0(f)\oplus\lambda_{h_0(a)}(h_0(f))$, where $+$ and $-$ are group addition and subtraction, respectively, taken in the group $G$. In other words, we have established that $h_0$ is also a homomorphism of $EMV$-algebras.

Denote by $\mathcal T$ the set of functions $f\in \mathcal T_0$ such that (1) there is $x \in M$ with $h_0(f)=x$, and (2) there is a characteristic function $a\in \mathcal T_0$ such that $f\le a$ and $h_0(f) \in \mathcal I(M)$. We assert that $\mathcal T$ is an $EMV$-tribe of fuzzy sets. Indeed, if $f,a\in \mathcal T$, where $f\le a$ and $a$ is a characteristic function, then $h_0(f)=x$, $b:=h_0(a)$ is an idempotent of $M$, and $x\le a$. Then $a-f \in \mathcal T_0$ and $a-f \le a$, and using a fact that $h_0$ is a homomorphism of $EMV$-algebras,  $a= f+(a-f)= a\oplus (a-f)$ implies $h_0(a-f)= \lambda_{h_0(a)}(f)\in \mathcal T$, i.e. $h_0(a-f)=\lambda_b(x)\in M$ which means $a-f \in \mathcal T$. Clearly $f,g \in \mathcal T$ implies $f\oplus g\in \mathcal T$, $f\vee g=\max\{f,g\}, f\wedge g = \min\{f,g\}\in \mathcal T$, whence, $\mathcal T$ is an $EMV$-tribe.

Now let $\{f_n\}$ be a sequence of functions from $\mathcal T$. Since $\mathcal T$ is closed under $\vee = \max$, we can assume that $\{f_n\}$ is non-decreasing. For each $n$, there is a characteristic function $a_n\in \mathcal T_0$ such that $f_n \le a_n$. We can choose $\{a_n\}$ to be also non-decreasing. Assume $h_0(f_n)=x_n \in M$ and $h_0(a_n)=b_n\in \mathcal I(M)$. Then $x=\bigvee_n x_n \in M$ and $b=\bigvee_nb_n \in \mathcal I(M)$. Define
$$
f(\omega)=\lim_n f_n(\omega),\quad a(\omega)=\lim_n a_n(\omega), \quad \omega \in \Omega.
$$
Then $a$ is a characteristic function with $f\le a$, and $h_0(a)=h_0(\bigvee_n a_n)=b$, $h_0(f)=x$ and $f \le a$, so that $a,f\in \mathcal T$.

Now let $\{f_n\}$ be a sequence of arbitrary functions from $\mathcal T$ and let each $f_n$ is dominated by a characteristic function $a\in \mathcal T$. Then $g_n:=f_1\oplus\cdots \oplus f_n = \min\{f_1+\cdots+f_n,a\} \in \mathcal T$ for each $n\ge 1$, $f=\lim_n g_n \in \mathcal T$, and $f=\min\{\sum_n f_n,a\}$. Consequently, $\mathcal T$ is an $EMV$-tribe of fuzzy functions.

Finally, if $h$ is the restriction of $h_0$ onto $\mathcal T$, then $h$ is a $\sigma$-homomorphism of $EMV$-algebras from $\mathcal T$ onto $M$ which completes the proof of the theorem.
\end{proof}

\section{Conclusion}

The main aim of the paper was to formulate and prove a variant of the Loomis--Sikorski theorem for $\sigma$-complete $EMV$-algebras. To do it, we have used some topological methods. The main complication is that an $EMV$-algebra does not possess a top element, in general. We have introduced the weak topology of state-morphisms and the hull-kernel topology of maximal ideals. We have shown that these spaces are homeomorphic, Theorem \ref{th:Bel9}, and they are compact iff the $EMV$-algebra possesses a top element. In general, these space are locally compact, completely regular and Hausdorff, Theorem \ref{th:Bel11}, and due to Corollary \ref{co:Baire}, they are Baire spaces. Nevertheless that an $EMV$-algebra $M$ does not possess a top element, due to the Basic Representation Theorem, it can be embedded into an $MV$-algebra $N$ as its maximal ideal and every element of $N$ either belongs to $M$ or is a complement of some element of $M$. Therefore, the one-point compactification of the state-morphisms space is homeomorphic to the state-morphism space of $N$, a similar result holds for the set of maximal ideals, Theorem \ref{th:comp}. The main result of the paper is the Loomis--Sikorski Theorem for $\sigma$-complete $EMV$-algebras, Theorem \ref{th:LS6} which says that every $\sigma$-complete $EMV$-algebra is a $\sigma$-epimorphic image of some $\sigma$-complete $EMV$-tribe which is a $\sigma$-complete $EMV$-algebra of fuzzy sets where all $EMV$-operations are defined by points. We have presented two proofs of the Loomis--Sikorski theorem, see also Theorem \ref{th:LS9}.

The presented paper enriches the class of \L ukasiewicz like algebraic structures where the top element is not assumed.

%%%%%%%%%%%%%%%Thebibliography%%%%%%%%%%%%%%%%%%%%%%%%%%%%%%%
%%%%%%%%%%%%%%%%%%%%%%%%%%%%%%%%%%%%%%%%%%%%%%%%%%%%%%%%%%%%%

\end{document}